\documentclass[nomath,reqno,a4paper,12pt]{amsart}

\makeatletter \let\th@plain\relax \makeatother

\usepackage{mystyle}

\usepackage[utf8]{inputenc}
\usepackage[T1]{fontenc}
\usepackage{lmodern}
\usepackage[british]{babel}
\usepackage{microtype}
\usepackage[headings]{fullpage}
\usepackage{parskip}

\usepackage[pdfencoding=auto,pdfusetitle]{hyperref}
\hypersetup{colorlinks=true, linkcolor=blue!30!black, citecolor=green!30!black, urlcolor=red!45!black}

\usepackage{microtype}
\usepackage{booktabs}

\usepackage{amsfonts,amssymb,bbm,stmaryrd}
\usepackage[scr=boondoxo,scrscaled=1.05]{mathalfa}
\usepackage{amsmath}
\usepackage[thmmarks, amsmath, amsthm]{ntheorem}
\theorempreskip{\parsep}
\AtBeginEnvironment{proof}{\vspace{\parsep}}

\usepackage{mathtools}
\usepackage[noabbrev]{cleveref}
\usepackage{centernot,float,subcaption,url,array,booktabs,colortbl}

\usepackage[shortlabels]{enumitem}
\setlist[enumerate,1]{label={\roman*)}}

\usepackage{tikz}
\usetikzlibrary{matrix,positioning}
\usetikzlibrary{calc}
\usetikzlibrary{arrows}
\tikzset{> =stealth}

\theoremstyle{plain}
\newtheorem{theorem}{Theorem}[section]
\newtheorem{lemma}[theorem]{Lemma}
\newtheorem{proposition}[theorem]{Proposition}
\newtheorem{corollary}[theorem]{Corollary}
\newtheorem{conjecture}[theorem]{Conjecture}

\theoremstyle{definition}
\newtheorem{definition}[theorem]{Definition}

\theoremstyle{remark}

\newtheorem{example}[theorem]{Example}

\newcommand{\mult}{\bigcdot}
\newcommand{\scalarMult}{\cdot}
\makeatletter
\newcommand*{\bigcdot}{}
\DeclareRobustCommand*{\bigcdot}{%
  \mathbin{\mathpalette\bigcdot@{}}%
}
\newcommand*{\bigcdot@scalefactor}{.5}
\newcommand*{\bigcdot@widthfactor}{1.15}
\newcommand*{\bigcdot@}[2]{%
  \sbox0{$#1\vcenter{}$}%
  \sbox2{$#1\cdot\m@th$}%
  \hbox to \bigcdot@widthfactor\wd2{%
    \hfil
    \raise\ht0\hbox{%
      \scalebox{\bigcdot@scalefactor}{%
        \lower\ht0\hbox{$#1\bullet\m@th$}%
      }%
    }%
    \hfil
  }%
}
\makeatother

\renewcommand{\epsilon}{\varepsilon}
\renewcommand{\phi}{\varphi}

\mathchardef\mhyphen="2D

\renewcommand{\land}{\mathrel{\wedge}}
\renewcommand{\lor}{\mathrel{\vee}}

\newcommand{\R}{\mathbb{R}}
\renewcommand{\C}{\mathbb{C}}

\newcommand{\Srpnsk}{\mathbb{S}}

\newcommand{\lowerRealsNonneg}{\overrightarrow{\R}_{\!\ge 0}}

\newcommand{\id}{\mathrm{id}}
\newcommand{\Cvar}{\mathscr{C}}

\newcommand{\Set}{\mathbf{Set}}

\newcommand{\Sup}{\mathbf{Sup}}
\newcommand{\DLat}{\mathbf{DLat}}
\newcommand{\Frm}{\mathbf{Frm}}
\newcommand{\Loc}{\mathbf{Loc}}
\newcommand{\OFrm}{\mathbf{OFrm}}
\newcommand{\OLoc}{\mathbf{OLoc}}
\newcommand{\Quant}{{\mathbf{Quant}_\top}}
\newcommand{\QuantGeneral}{\mathbf{Quant}}
\newcommand{\Dcpo}{\mathbf{DCPO}}

\newcommand{\LocSRng}{\mathbf{LocCRig}}
\newcommand*{\Mod}[1]{{#1}\raisebox{0.1ex}{-\kern1.0pt}\mathbf{Mod}}

\newcommand{\CComon}{\mathrm{CComon}}

\newcommand{\op}{{^\mathrm{\hspace{0.5pt}op}}}

\newcommand*{\dirsup}[1][\hspace{0.75ex}]{\operatorname*{\bigvee{}^{\!\uparrow}_{\hspace{-0.2ex} #1}}}
\newcommand\twoheaduparrow{\rotatebox[origin=c]{90}{$\twoheadrightarrow$}}
\newcommand\twoheaddownarrow{\rotatebox[origin=c]{90}{$\twoheadleftarrow$}}

\renewcommand{\O}{\mathcal{O}}

\DeclareMathOperator{\Hom}{Hom}
\DeclareMathOperator{\Spec}{Spec}
\DeclareMathOperator{\Idl}{Idl}
\DeclareMathOperator{\Rad}{Rad}
\newcommand{\OPAI}{\mathrm{OPAI}}
\newcommand{\OPMAI}{\mathrm{OPMAI}}

\newcommand{\stabpos}{\pi}

\newcommand{\addquot}{\mathfrak{s}}
\newcommand{\rad}{\rho}
\newcommand{\Sats}{\mathscr{S}}
\newcommand{\MonIdl}{\mathscr{M}}

\newcommand{\doublePow}{\mathbb{P}}
\newcommand{\diam}{\mathrm{diam}}
\newcommand{\ev}{\mathsf{ev}}

\DeclarePairedDelimiter\abs{\lvert}{\rvert}
\DeclarePairedDelimiter\norm{\lVert}{\rVert}
\makeatletter
\let\oldabs\abs
\def\abs{\@ifstar{\oldabs}{\oldabs*}}
\let\oldnorm\norm
\def\norm{\@ifstar{\oldnorm}{\oldnorm*}}
\makeatother

\pgfdeclarelayer{background}
\pgfsetlayers{background,main}
\tikzset{dot/.style={circle,draw=black,fill=black,minimum size=1mm,inner sep=0mm}}

\newcommand\fillbackground[3][black!10]{%
\begin{pgfonlayer}{background}
\fill[#1] (#2) rectangle (#3);
\end{pgfonlayer}
}

\title{The spectrum of a localic semiring}

\author[G. Manuell]{Graham Manuell}
\address{Centre for Mathematics\\ University of Coimbra}
\email{graham@manuell.me}

\date{December 2021}

\subjclass[2010]{54B35, 06D22, 54H13, 13J99, 54B30, 13A15, 06F07, 03F65}
\keywords{spectra, frame, quantale, Zariski topology, closed prime ideal}

\begin{document}

\begin{abstract}
A number of spectrum constructions have been devised to extract topological spaces from algebraic data.
Prominent examples include the Zariski spectrum of a commutative ring, the Stone spectrum of a bounded distributive lattice,
the Gelfand spectrum of a commutative unital C*-algebra and the Hofmann--Lawson spectrum of a continuous frame.

Inspired by the examples above, we define a spectrum for localic semirings. We use arguments in the symmetric monoidal category of suplattices
to prove that, under conditions satisfied by the aforementioned examples, the spectrum can be constructed as
the frame of overt weakly closed radical ideals and that it reduces to the usual constructions in those cases. Our proofs are constructive.

Our approach actually gives `quantalic' spectrum from which the more familiar localic spectrum can then be derived. For a discrete ring this yields the quantale of ideals
and in general should contain additional `differential' information about the semiring.
\end{abstract}

\maketitle
\thispagestyle{empty}

\setcounter{section}{-1}
\section{Introduction}

A spectrum constructions is a way to assign topological spaces to certain algebraic structures. The original example is Stone's spectrum of a Boolean algebra \cite{StoneBoolAlgs} or
general (bounded) distributive lattice \cite{StoneDistLattices}. Other important spectra include the so-called Zariski spectrum of a commutative ring (essentially introduced by Jacobson in \cite{JacobsonPrimitive}),
Gelfand's spectrum of a commutative unital C*-algebra \cite{GelfandNormedSpectrum}
and Hofmann and Lawson's spectrum of a distributive continuous lattice \cite{hofmann1978spectral}.

The above spectra have a number of striking similarities. In particular, the points of the Stone and Zariski spectra both correspond to prime ideals of the semirings in question.
The points of the Gelfand and Hofmann--Lawson spectra might at first appear to be slightly different: in the former case they are usually described as the \emph{maximal} ideals
and in the latter case as the prime \emph{elements}.
However, in these cases the semirings come equipped with a natural topology which we should not ignore. The maximal ideals of a C*-algebra coincide with the closed prime ideals with respect to the norm topology and the prime elements of a continuous distributive lattice correspond to closed prime ideals with respect to the Scott topology.
Furthermore, for all of these cases, the topology on the spectrum is given by the ``hull-kernel topology'' and the opens are in bijection with the closed radical ideals.

Further evidence for the depth of these relationships is that remain true when we work constructively.
In this setting it is better use locales instead of topological spaces.
The algebraic structures in all of our examples can be considered to be \emph{localic semirings}.
While the points of the Zariski spectrum are now prime anti-ideals instead of prime ideals, their complementary sublocales are closed prime ideals as before.
In summary we obtain the following table.

\begin{table}[H]
\begin{tabular}{>{\raggedright\arraybackslash}m{38mm}>{\raggedright\arraybackslash}m{23mm}>{\raggedright\arraybackslash}m{35mm}>{\raggedright\arraybackslash}m{35mm}}
  \toprule
  \textbf{Class of semiring} & \textbf{Spectrum} & \textbf{Opens} & \textbf{Points} \\
  \midrule
  Commutative rings & Zariski & Radical ideals & Prime anti-ideals \\
  \rowcolor{gray!20}
  Distributive lattices & Stone & Ideals & Prime filters \\
  Comm.\ C*-algebras & Gelfand & Overt weakly closed ideals \cite{ConstructiveGelfandNonunital} & Closed prime ideals \\
  \rowcolor{gray!20}
  Continuous frames & Hofmann--Lawson & Scott-closed ideals & Scott-open prime filters \\
  \bottomrule
\end{tabular}
\end{table}

In each case the opens correspond to `overt weakly closed radical ideals' and the points are given by closed prime ideals
(or equivalently open prime anti-ideals). We will define a general notion of spectrum of a localic semiring as a classifying locale for closed prime ideals
and the frame of overt weakly closed radical ideals and prove that these coincide under assumptions satisfied by our examples.

The intuition behind considering the closed prime ideals is that when viewing the semiring as consisting of functions on a putative spectrum,
the closed prime ideals correspond to the places where these functions are zero --- at least if we require that being nonzero is an open condition.

Our approach will make use of the quantale of (overt, weakly closed) ideals. We use this to construct the frame of radical ideals, but it is also of interest in its own right.
At least in the discrete case, this quantale contains additional `differential' information, which can capture some of the ideas around `repeated roots' and singularities in algebraic geometry.

A number of the results of this paper, further intuition and a significant amount of additional background material can be found in my PhD thesis \cite{ManuellThesis}.

\subsection*{Outline}

After covering some background material, we will define a generalised notion of presentation of a frame. This will allow us to extend the usual presentation of the Zariski spectrum to general localic semirings.
However, the frame defined by such a generalised presentation might not exist in general. It will be useful to ensure it exists under certain conditions and to provide an explicit construction in this case.

In fact, we define a more general kind of spectrum, giving a quantale instead of just a frame. This generalises the quantale of ideals of a ring and can be used to recover the localic spectrum above.

Before tackling the semiring spectrum, we consider the simpler case of the spectrum of a commutative localic monoid. We define the notions of saturated opens and overt weakly closed monoid ideals
and prove that under certain assumptions they form dual objects in the monoidal category of suplattices. We can use this result to prove that quantale of monoid ideals is the quantic monoid spectrum.

We then modify this monoid result to solve the semiring case. We make use of the duality between saturated opens and monoid ideals, but need to modify the additive structure to make it compatible with these.
We show that the quantic spectrum is the quantale of overt weakly closed (semiring) ideals and the localic spectrum is the frame of radical ideals.
Finally, we show that all of our core examples are special cases of our main result.
We conclude with some unusual examples and suggest how we might prove the spectrum does not always exist.

\section{Background}

The arguments in this paper are constructive in the sense that they hold internal to any elementary topos with natural numbers object.

Recall that a \emph{semiring} (or \emph{rig}) is a set equipped with an `additive' commutative monoid structure and a `multiplicative' monoid structure such that
the multiplication is bilinear with respect to the additive structure. All semirings in this paper are assumed to be commutative.

\subsection{Suplattices, quantales and frames}

A \emph{suplattice} is a poset admitting arbitrary joins. We write $0$ for the least element of a suplattice and $\top$ for the largest element.
As objects, suplattices are the same as complete lattices, but morphisms of suplattices need only preserve joins.
We denote the category of suplattice by $\Sup$.

The category of suplattices is symmetric monoidal closed with the tensor product of suplattices defined similarly to abelian groups as the codomain of the the initial bilinear map $L \times M \to L \otimes M$.
The unit is given by the lattice of truth values $\Omega$. The internal hom is the suplattice $\hom(L, M)$ of suplattice homomorphisms from $L$ to $M$ with the pointwise order.
In this paper $\hom(-,-)$ will always refer to the internal hom in $\Sup$.

A \emph{quantale} is a monoid in $\Sup$. In this paper we adopt the convention that every quantale is commutative. We write $\QuantGeneral$ for the category of (commutative) quantales.

Finite coproducts in $\QuantGeneral$ are given by tensor products of their underlying suplattices and the obvious multiplication.
The suplattice $\Omega$ becomes the initial quantale with the multiplication given by the isomorphism $\Delta\colon \Omega \otimes \Omega \xrightarrow{\sim} \Omega$ sending $p \otimes q$ to $p \wedge q$.
The unique quantale map ${!}\colon \Omega \to Q$ induces a `quantale module' structure on $Q$ giving a scalar multiplication $p \scalarMult q = {!}(p) \mult q = \bigvee \{q \mid p\}$.\footnote{Here $\{q \mid p\}$ is a subsingleton which contains the element $q$ if and only if $p$ holds.}

Most of our quantales will be \emph{two-sided} --- that is, the unit $1$ coincides with
the largest element $\top$, or equivalently, $xy \le x, y$ for all $x,y$. The category $\Quant$ of two-sided quantales is reflective and coreflective in $\QuantGeneral$,
with the reflection of a quantale given by its largest two-sided quotient. This quotient can be obtained as the fixed points of the closure operator $a \mapsto a \top$.
This closure operator is an example of a \emph{nucleus}, which can be used to describe quotient quantales more generally.

A \emph{frame} is precisely a two-sided quantale with idempotent multiplication. In this case, the multiplication is given by meet. The category of frames is denoted $\Frm$
and is both reflective and coreflective in $\Quant$. This reflection is called the \emph{localic reflection} and the unit is a quotient map as above.

For more background on suplattices and quantales see \cite{galoisTheoryGrothendieck} and \cite{RosenthalQuantales}.

\subsection{Frames and locales}

Frames can be understood as generalised lattices of open sets and are more amenable to a constructive treatment than topological spaces.
A continuous map of topological spaces induces a morphism between the frames of open sets in the opposite direction and so when we wish to emphasise
the spatial intuition, we will work with the category $\Loc = \Frm\op$. To avoid any confusion between these categories,
we will explicitly write $\O X$ for the frame corresponding to a locale $X$ and $f^*$ for the frame homomorphism corresponding to a locale map $f$.

We say a locale corresponding to a topological space is \emph{spatial}. In good cases this space may be recovered from the locale with its points being given by the locale maps from the terminal locale $1$.
The free frame on one generator is spatial and its corresponding topological space is given by the set of truth values $\Omega$ equipped with the topology generated by the subbasic open $\{\top\}$.
This is the \emph{Sierpiński space} $\Srpnsk$ and should not be confused with the terminal locale, which has $\Omega$ as its frame of opens.
A \emph{discrete locale} is a locale corresponding to a discrete topological space $X$. Its opens are given by the subsets of $X$.

The pointfree analogue of subspaces are called \emph{sublocales} and correspond to quotient frames. Elements of a frame give rise to \emph{open sublocales}. Unlike subsets in constructive mathematics,
open sublocales always have complements, which are called \emph{closed sublocales}.

The category of frames is order-enriched with the obvious pointwise order on morphisms and this induces a similar enrichment on $\Loc$. This is related to the specialisation
preorder of topological spaces.

We now describe some particular aspects of pointfree topology which will be of use to us. For more information see \cite{PicadoPultr} and \cite[Part C]{Elephant2}.

A \emph{localic monoid} is a monoid object in the category of locales or equivalently, a comonoid object in the monoidal category of frames where the tensor product is given by the coproduct.
If $M$ is a localic monoid, we call the corresponding comultiplication map $\mu_\times\colon \O M \to \O M \oplus \O M$ and the counit map $\epsilon_1\colon \O M \to \Omega$.
Localic monoids can be thought of as being like topological monoids, with the points corresponding to the `elements of the monoid', but as is always the case with pointfree topology, the points do not tell the whole story. On the other hand, note that the opens do not have the structure of a monoid, but of a comonoid.

A \emph{localic semiring} is a semiring object in $\Loc$.
Given a localic semiring $R$, we write $\epsilon_0, \epsilon_1\colon \O R \to \Omega$ for the frame homomorphisms corresponding to the additive and multiplicative unit respectively.
We write $\mu_+, \mu_\times\colon \O R \to \O R\oplus \O R$ for the frame homomorphisms corresponding to addition and multiplication.
The category of localic semirings will be called $\LocSRng$.

In a constructive setting, it is more useful to know that a set is \emph{inhabited} (i.e. it contains an element), instead of simply nonempty.
Similarly, there is a stronger version of nontriviality for frames.
\begin{definition}
 An element $a$ of a frame $L$ is said to be \emph{positive} if whenever $\bigvee I \ge a$ then $I$ is inhabited.
 We say $L$ is positive if $1 \in L$ is positive.
\end{definition}

This leads to a concept which is invisible in the classical setting.

\begin{definition}
 We say a frame is \emph{overt}\footnote{Some sources call such frames \emph{open}, but this leads to ambiguity when discussing sublocales.}
 or \emph{locally positive} if it has a base of positive elements.
\end{definition}

Under the assumption of excluded middle, every locale is overt.
More generally, overtness has strong links to openness.

\begin{lemma}
 A locale $X$ is overt if and only if the unique locale map ${!}\colon X \to 1$ is open
 if and only if the product projection $\pi_1\colon Y \times X \to Y$ is open for all $Y$.
 
 In this case, the frame homomorphism ${!}\colon \Omega \to \O X$ has a left adjoint $\exists\colon \O X \to \Omega$
 such that $\exists(a) = \top$ if and only if $a$ is positive.
 The left adjoint $(\iota_1)_!$ of the coproduct injection $\iota_1\colon \O Y \to \O Y \oplus \O X$ is given by $\O Y \otimes \exists$ (up to composition with $\O Y \otimes \Omega \xrightarrow{\sim} \O Y$).
\end{lemma}

We denote the full subcategories of overt frames and overt locales by $\OFrm$ and $\OLoc$ respectively.
It can be shown that every spatial locale is overt and the category of overt locales is closed under finite products in $\Loc$.
In particular, this implies that overt localic monoids or semirings coincide with monoids or semirings in $\OLoc$.

Overt locales should be thought of as locales that we can `existentially quantify' over. For example, if we consider $a \in \O X$ as a proposition,
the map $\exists\colon \O X \to \Omega$ tells us whether `there is something' satisfying $a$. In fact, this intuition can be made precise.
There is an obvious functor $\OLoc\op \to \Frm \to \DLat$, where $\DLat$ is the category of distributive lattices.
This is a \emph{coherent hyperdoctrine without equality} which yields an internal logic for manipulating open propositions on overt locales, admitting disjunction, conjunction and existential quantification. (See \cite{PittsLogic} for details about logic in coherent hyperdoctrines.)

Finally, in constructive mathematics, the notion of closedness bifurcates to give closed sublocales and \emph{weakly closed} (or \emph{fibrewise closed}) sublocales.
The \emph{overt} weakly closed sublocales of $X$ are in bijection with the suplattice homomorphisms from $\O X$ to $\Omega$ (see \cite{BungeFunkLowerPowerlocale} for details).
Given an overt sublocale $V$ of $X$, we can ask if an open $a \in \O X$ restricts to a positive open in $V$. In this case, we say $V$ \emph{meets} $a$ and write $V \between a$.
This defines a suplattice homomorphism from $\O X$ into $\Omega$, from which the weak closure of $V$ may be recovered.
The overt weakly closed sublocales of a discrete locale are precisely the opens (i.e. the subsets) and here we recover the notation $S \between T$ for subsets $S$ and $T$ having inhabited intersection.

If $f\colon X \to Y$ is a locale map, then $\hom(f^*, \Omega)$ induces a map from overt weakly closed sublocales of $X$ to overt weakly closed sublocales of $Y$
which can be interpreted as taking the weak closure of the image. More generally, if $V$ is an overt sublocale of $X$, the image $V'$ of $V$ is overt and
$V' \between a \iff V \between f^*(a)$.

\subsection{Dcpos}

Recall that a poset $D$ is called \emph{directed} if every finitely-indexed subset $F \subseteq D$ has an upper bound in $D$.
A \emph{dcpo} is a poset which admits joins of its directed subsets. A map of dcpos is a monotone function which preserves these directed joins and we denote the category of dcpos by $\Dcpo$.

A dcpo comes with a natural topology, the \emph{Scott topology}, consisting of the upsets $U$ satisfying $\bigvee D \in U \implies D \between U$ for all directed sets $D$.
Morphisms of dcpos are continuous with respect to this topology.

Another important class of subsets is given by the \emph{Scott-closed sets}, which are downsets closed under directed joins.
These are not simply the complements of Scott-open sets; we will explain their relationship to the Scott topology in \cref{prop:scott_closed_weakly_closed}.

The \emph{way-below} relation on a dcpo $P$ is defined so that $a \ll b$ if whenever $D \subseteq P$ is directed and $b \le \bigvee D$ then there exists some $d \in D$ such that $a \le d$.
We write $\twoheaddownarrow b$ for the set $\{ a \in P \mid a \ll b \}$ and $\twoheaduparrow a$ for the set $\{ b \in P \mid a \ll b \}$.

\begin{definition}
 A dcpo $P$ is \emph{continuous} if for all $b \in P$, the set $\twoheaddownarrow b$ is directed and has $b$ as its supremum.
\end{definition}

\begin{lemma}\label{prop:base_scott_top_continuous_dcpo}
 In a continuous dcpo, the sets of the form $\twoheaduparrow a$ form a base for the Scott topology.
\end{lemma}

\begin{lemma}\label{prop:scott_closure_continuous_dcpo}
 The Scott-closure of a downset $S$ in a continuous dcpo is given by the set of directed joins of elements of $S$.
 (That is, the closure is obtained in a single step.)
\end{lemma}

A locale $X$ is exponentiable if and only if $\O X$ is continuous as a dcpo. We say such a locale is \emph{locally compact}.
In this case, $\Srpnsk^X$ is a spatial locale whose corresponding topological space is $\O X$ with the Scott topology.

\subsection{Dualisable suplattices}

An object $A$ in a monoidal category $\Cvar$ is (left) \emph{dualisable} if it is a (left) adjoint when $\Cvar$ is viewed as a one-object bicategory. Its right adjoint is called the (right) \emph{dual} of $A$
and is written $A^*$.

Explicitly, $A$ has a (right) dual $A^*$ if there exist maps $\eta\colon I \to A^* \otimes A$ and $\epsilon\colon A \otimes A^* \to I$
such that the following diagrams commute.
\begin{center}
 \begin{minipage}{.45\textwidth}
  \centering
  \begin{tikzpicture}[node distance=3.5cm, auto]
    \node (A) {$A \otimes (A^* \otimes A)$};
    \node (B) [right of=A] {$(A \otimes A^*) \otimes A$};
    \node (C) [below=3cm of A] {$A \otimes I$};
    \node (D) [right of=C] {$I \otimes A$};
    \path (C) -- node[anchor=center] (mid) {$A$} (D);
    \draw[->] (A) to node {$\sim$} (B);
    \draw[->] (C) to node {$A \otimes \eta$} (A);
    \draw[->] (C) to node {$\sim$} (mid);
    \draw[<-] (mid) to node {$\sim$} (D);
    \draw[->] (B) to node {$\epsilon \otimes A$} (D);
  \end{tikzpicture}
 \end{minipage}
 \begin{minipage}{.45\textwidth}
  \centering
  \begin{tikzpicture}[node distance=3.5cm, auto]
    \node (A) {$(A^* \otimes A) \otimes A^*$};
    \node (B) [right of=A] {$A^* \otimes (A \otimes A^*)$};
    \node (C) [below=3cm of A] {$I \otimes A^*$};
    \node (D) [right of=C] {$A^* \otimes I$};
    \path (C) -- node[anchor=center] (mid) {$A^*$} (D);
    \draw[->] (A) to node {$\sim$} (B);
    \draw[->] (C) to node {$\eta \otimes A^*$} (A);
    \draw[->] (C) to node {$\sim$} (mid);
    \draw[<-] (mid) to node {$\sim$} (D);
    \draw[->] (B) to node {$A^* \otimes \epsilon$} (D);
  \end{tikzpicture}
 \end{minipage}
\end{center}
Here $I$ is the unit of the monoidal category and the unnamed isomorphisms are the associator (or its inverse) and the left and right unitors as appropriate.

These conditions can be expressed more evocatively in the language of string diagrams (see \cite{marsden2014category,selinger2010survey}),
which we lay out vertically from bottom to top.
\begin{center}
\begin{minipage}{0.49\textwidth}
\begin{equation*}
\begin{tikzpicture}[scale=0.5,baseline={([yshift=-0.5ex]current bounding box.center)}]
\path coordinate[dot, label=below:$\epsilon$] (epsilon) ++(1,-1) coordinate (a) ++(1,-1) coordinate[dot, label=above:$\eta$] (eta)
 ++(1,1) coordinate (b) ++(0,2) coordinate[label=above:$A$] (tr)
 (epsilon) ++(-1,-1) coordinate (c) ++(0,-2) coordinate[label=below:$A$] (bl);
\draw (bl) -- (c) to[out=90, in=180] (epsilon) to[out=0, in=90] (a) to[out=-90, in=180] (eta) to[out=0, in=-90] (b) -- (tr);
\fillbackground{$(bl) + (-0.5,0)$}{$(tr) + (0.5,0)$};
\end{tikzpicture}
\enspace=\enspace
\begin{tikzpicture}[scale=0.5,baseline={([yshift=-0.5ex]current bounding box.center)}]
\path coordinate[label=below:$A$] (b) ++(0,4) coordinate[label=above:$A$] (t);
\draw (b) -- (t);
\fillbackground{$(t) + (-1,0)$}{$(b) + (1,0)$};
\end{tikzpicture}
\end{equation*}
\end{minipage}
\begin{minipage}{0.49\textwidth}
\begin{equation*}
\begin{tikzpicture}[scale=0.5,baseline={([yshift=-0.5ex]current bounding box.center)}]
\path coordinate[dot, label=above:$\eta$] (eta) ++(1,1) coordinate (a) ++(1,1) coordinate[dot, label=below:$\epsilon$] (epsilon)
 ++(1,-1) coordinate (b) ++(0,-2) coordinate[label=below:$A^*$] (br)
 (eta) ++(-1,1) coordinate (c) ++(0,2) coordinate[label=above:$A^*$] (tl);
\draw (tl) -- (c) to[out=-90, in=180] (eta) to[out=0, in=-90] (a) to[out=90, in=180] (epsilon) to[out=0, in=90] (b) -- (br);
\fillbackground{$(tl) + (-0.5,0)$}{$(br) + (0.5,0)$};
\end{tikzpicture}
\,=\enspace
\begin{tikzpicture}[scale=0.5,baseline={([yshift=-0.5ex]current bounding box.center)}]
\path coordinate[label=below:$A^*$] (b) ++(0,4) coordinate[label=above:$A^*$] (t);
\draw (b) -- (t);
\fillbackground{$(t) + (-1,0)$}{$(b) + (1,0)$};
\end{tikzpicture}
\end{equation*}
\end{minipage}
\end{center}
So the maps $\epsilon$ and $\eta$ allow us to `turn corners' in string diagrams and the identities can be thought of as saying we can `pull the wires straight'.
We will usually suppress the dots and labels for $\epsilon$ and $\eta$ in these situations as they can be readily understood from context.

In a symmetric monoidal category left and right dualisability are equivalent and so we simply call such an object \emph{dualisable}.

If the category $\Cvar$ is monoidal closed, then $A \otimes (-) \dashv \hom(A, -)$. So by uniqueness of adjoints, $A^* \otimes (-) \cong \hom(A,-)$ whenever $A^*$ exists.
In particular, $A^* \cong A^* \otimes I \cong \hom(A,I)$. Moreover, if we take $A^* = \hom(A,I)$, the counit $\epsilon\colon A \otimes A^* \to I$ is given by the (flipped)
evaluation map $A \otimes \hom(A,I) \to I$.

In $\Sup$ the dualisable objects are related to the notion of a \emph{dual basis}.

\begin{definition}
 Let $L$ be a suplattice. We say that a family $(r_x)_{x \in X}$ of elements of $L$ and a family $(\sigma_x)_{x \in X}$ of elements of $\hom(L, \Omega)$ form a \emph{dual basis} for $L$
 if $a = \bigvee_{x \in X} \sigma_x(a) \scalarMult r_x$ for all $a \in L$.
\end{definition}

The following modification of continuity is also of interest.
\begin{definition}
 We define the totally below relation on a suplattice by setting $a \lll b$ if whenever $\bigvee S \ge b$ then there exists some $s \in S$ such that $s \ge a$.
 We say a suplattice is \emph{supercontinuous} if every element is the join of the elements totally below it.
\end{definition}

We now have the following lemma (see \cite[Section 1.6]{ManuellThesis}).
\begin{lemma}\label{prop:supercontinuous_vs_dualisable}
 Let $L$ be a suplattice. The following conditions are equivalent:
 \begin{enumerate}
  \item $L$ admits a dual basis,
  \item $L$ is dualisable,
  \item $L$ is supercontinuous.
\end{enumerate}
\end{lemma}
Assuming the axiom of choice, these are precisely the completely distributive lattices, so dualisable suplattices are also called \emph{constructively completely distributive} \cite{Fawcett1990completeDistributivity}.
In particular, such a suplattice is always a frame. Moreover, since $0 \lll a$ precisely when $a$ is positive, it is not hard to show such a supercontinuous frame is always overt.

\subsection{The Zariski spectrum}

The \emph{Zariski spectrum} of a (discrete) ring $R$ is the classifying locale of the geometric theory of the prime anti-ideals of $R$.
Each element $f \in R$ gives a proposition $\overline{f}$ on the prime anti-ideals of $R$. We interpret $\overline{f}$ holding for a prime anti-ideal as meaning that $f$ lies in it.
The definition of a prime anti-ideal is then given by the following axioms.
\begin{align*}
 \overline{0}    &\mathmakebox[4.5ex][c]{\vdash} \bot \\
 \overline{f+g}  &\mathmakebox[4.5ex][c]{\vdash} \overline{f} \lor \overline{g} \\
 \top &\mathmakebox[4.5ex][c]{\vdash} \overline{1} \\
 \overline{fg}   &\mathmakebox[4.5ex][c]{\dashv\vdash} \overline{f} \land \overline{g}
\end{align*}
More geometrically, we can imagine the elements of $R$ as functions defined on the spectrum where $f$ is \emph{cozero} (that is, nonzero in a positive sense) at a point
if and only if it lies in the corresponding anti-ideal.

Explicitly, the frame obtained from this propositional theory has the following presentation.
\[\langle \overline{f} : f \in R \mid \overline{0} = 0,\, \overline{f+g} \le \overline{f} \vee \overline{g},\,
\overline{1} = 1,\, \overline{fg} = \overline{f} \wedge \overline{g} \rangle\]
This can be shown to correspond to the lattice $\Rad(R)$ of radical ideals of $R$.

A similar presentation of two-sided quantales
\[\langle \overline{f} : f \in R \mid \overline{0} = 0,\, \overline{f+g} \le \overline{f} \vee \overline{g},\,
\overline{1} = 1,\, \overline{fg} = \overline{f} \mult \overline{g} \rangle\]
gives the suplattice $\Idl(R)$ of ideals of $R$ with a quantale operation given by multiplication of ideals.
Then applying the reflection of two-sided quantales into frames gives the frame of radical ideals described above.

For the quantale $\Idl(R)$, the generators corresponding to $f$ and $f^2$ are different. Geometrically, this can be thought of as arising from the fact that while $f$ and $f^2$ vanish in the same locations,
if $f$ vanishes to first order, then $f^2$ will vanish to second order --- that is, both $f$ and its first derivatives will be zero.

\section{Spectra from generalised presentations}

We could perhaps define the spectrum of a localic semiring directly in terms of overt weakly closed radical ideals, but this might not be completely convincing.
It would be better to define the spectrum by a universal property. We will take the presentation of the Zariski spectrum as our starting point.

We wish to generalise the presentation of the Zariski spectrum from discrete rings to localic semirings. No modifications are necessary to handle the case of discrete semirings,
but the non-discrete case requires some more care. In the presentation of the Zariski spectrum of $R$, the generators are given by points of $R$, but this is obviously not appropriate for a non-spatial localic semiring.
We require a generalised notion of presentation that allows the generators and relations to be locales.

A usual presentation of a frame $L$ can be viewed as expressing the frame as the coequaliser of maps between free frames $\langle \mathsf{R}\rangle \rightrightarrows \langle \mathsf{G}\rangle \twoheadrightarrow L$,
where $\mathsf{G}$ is the set of generators and $\mathsf{R}$ indexes the relations. From the localic perspective, the free frame on $\mathsf{G}$ corresponds to the $\mathsf{G}^\text{th}$ power of Sierpiński space $\Srpnsk$
and so the presentation corresponds to an equaliser $L \hookrightarrow \Srpnsk^G \rightrightarrows \Srpnsk^\mathsf{R}$ (where we reuse the variable $L$ for the corresponding locale).

We now reinterpret $\mathsf{G}$ and $\mathsf{R}$ as discrete locales and the powers of $\Srpnsk$ as exponentials. This allows us to replace $\mathsf{R}$ and $\mathsf{G}$ with any exponentiable locales. However, since our localic semirings might
not be locally compact, this is not yet sufficiently general for our purposes.

We can circumvent the nonexistence of exponentials by passing temporarily to the presheaf category $\Set^{\Loc\op}$ and performing the calculation there.
In this category, the exponential $\Srpnsk^\mathsf{G}$ is given by $\Hom_\Loc((-) \times \mathsf{G}, \Srpnsk)$.
We now consider the equaliser of a pair of natural transformations $F \hookrightarrow \Hom_\Loc((-) \times \mathsf{G}, \Srpnsk) \rightrightarrows \Hom_\Loc((-) \times \mathsf{R}, \Srpnsk)$.
The resulting functor is not always representable, but when it is, we say the representing object $L$ is presented by the \emph{generalised presentation}.
Observe that when the exponentials do exist, this coincides with the less general kind of presentation mentioned above.

It is shown in \cite{VickersTownsendDoublePowerlocale} that natural transformations from $\Hom_\Loc((-) \times \mathsf{G}, \Srpnsk)$ to $\Hom_\Loc((-) \times \mathsf{R}, \Srpnsk)$
are in bijection with dcpo morphisms from $\O \mathsf{G}$ to $\O \mathsf{R}$ (by taking the component of each natural transformation at $1$).
Thus, a generalised presentation may be described by a locale $\mathsf{G}$ of generators and a pair of dcpo morphisms from $\mathsf{G}$ into another locale $\mathsf{R}$ describing the relations.

Now mimicking the Zariski spectrum, we define the spectrum of a localic semiring $R$ by a generalised presentation where $R$ is the locale of generators.
In the Zariski case we have one relation involving $0$, one involving $1$, an $(R\times R)$-indexed family involving addition and an $(R\times R)$-indexed family involving multiplication, so the
relations are indexed by the set $1 \sqcup 1 \sqcup R\times R \sqcup R\times R$. Similarly, in our case the relations are indexed by the locale given by the frame
$\Omega \times \Omega \times (\O R \oplus \O R) \times (\O R \oplus \O R)$.
The relations are described by a parallel pair of dcpo morphisms, which decompose into four pairs of dcpo morphisms $\O R \rightrightarrows \Omega$, $\O R \rightrightarrows \Omega$,
$\O R \rightrightarrows \O R \oplus \O R$ and $\O R \rightrightarrows \O R \oplus \O R$. The first pair is $\epsilon_0$ and the constant function $0$. The second pair is $\epsilon_1$ and the constant function $1$.
The third is $\mu_+ \vee \iota_1 \vee \iota_2$ and $\iota_1 \vee \iota_2$, where $\iota_{1,2}\colon R \to R\oplus R$ are the coproduct injections. The fourth pair is $\mu_\times$ and $\iota_1 \wedge \iota_2$.
(Note that the third pair is setting $\mu_+ \vee \iota_1 \vee \iota_2$ equal to $\iota_1 \vee \iota_2$,
which is the same as requiring $\mu_+$ to be less than or equal to $\iota_1 \vee \iota_2$.)

The resulting equaliser for this generalised presentation can be explicitly described as a functor $\OPAI_R\colon \Loc\op \to \Set$ whose action on objects is given by
\begin{align*}
 \OPAI_R\colon X \mapsto \big\{ u \in \O X \oplus \O R \mid {}
 & (\O X \oplus \epsilon_0)(u) = 0,\ (\O X \oplus \epsilon_1)(u) =  1, \\ 
 & (\O X \oplus \mu_+)(u) \le (\O X \oplus \iota_1)(u) \vee (\O X \oplus\iota_2)(u), \\
 &  (\O X \oplus \mu_\times)(u) = (\O X \oplus\iota_1)(u) \wedge (\O X \oplus\iota_2)(u) \big\}
\end{align*}
where we use elements of the frame $\O X \oplus \O R$ in place of locale maps $X \times R \to \Srpnsk$.
If $f\colon Y \to X$ is a map of locales, then $\OPAI_R(f)$ sends $u \in \OPAI_R(X)$ to $(f^* \oplus \O R)(u)$.
We call this $\OPAI$ since it gives the open prime anti-ideals (or equivalently the closed prime ideals) of $R$ `fibred over $X$'.
In particular, note that $\OPAI_R(1)$ is the set of open prime anti-ideals of $R$.

We call the representing object of $\OPAI_R$ the \emph{(localic) spectrum} of $R$ (in the case it exists).
Note that $\OPAI_R$ is also functorial in $R$ and so we obtain a partial (ana)functor $\Spec\colon \LocSRng\op \rightharpoonup \Loc$ by parametrised representability.

This definition is easily extended from locales to two-sided quantales resulting in a functor $\overline{\OPAI}_R\colon \Quant \to \Set$
given by
\begin{align*}
 \overline{\OPAI}_R\colon Q \mapsto \big\{ u \in Q \oplus \O R \mid {}
 & (Q \oplus \epsilon_0)(u) = 0,\ (Q \oplus \epsilon_1)(u) =  1, \\ 
 & (Q \oplus \mu_+)(u) \le (Q \oplus \iota_1)(u) \vee (Q \oplus\iota_2)(u), \\
 &  (Q \oplus \mu_\times)(u) = (Q \oplus\iota_1)(u) \mult (Q \oplus\iota_2)(u) \big\}.
\end{align*}
When it exists, the representing object of $\overline{\OPAI}_R$ is called the \emph{quantic spectrum} of $R$. In this case, the localic spectrum is obtained as the localic reflection of the quantic spectrum.

In the following sections we will give conditions on $R$ for this spectrum to exist and describe its form under these assumptions.

\section{Monoid ideals and saturated elements}

Before we try to construct the spectrum of a localic semiring,
let us consider the simpler case of the spectrum of a commutative localic \emph{monoid} $M$.

We start by defining a variant of $\overline{\OPAI}$ which only involves the multiplicative relations.

\begin{definition}
The functor $\overline{\OPMAI}_M\colon \Quant \to \Set$ is defined on objects by
\begin{align*}
 \overline{\OPMAI}_M\colon Q \mapsto \big\{ u \in Q \oplus \O M \mid {}
 &  (Q \oplus \epsilon_1)(u) =  1, \\
 &  (Q \oplus \mu_\times)(u) = (Q \oplus\iota_1)(u) \mult (Q \oplus\iota_2)(u) \big\}
\end{align*}
and acts on morphisms in the obvious way.
\end{definition}

Recall that a \emph{monoid ideal} $I$ in a commutative monoid $M$ is a subset of $M$ for which $x \in I, y \in M \implies xy \in I$.
The above functor gives the (quantic) `open' prime \emph{monoid} anti-ideals of $M$.

In the discrete case, the quantic monoid spectrum is given by the quantale of monoid ideals on $M$, which we denote by $\MonIdl M$.
This is isomorphic to the free two-sided quantale on the monoid $M$, explicitly $\langle \overline{f} : f \in M \mid \overline{1} = 1,\, \overline{fg} = \overline{f} \mult \overline{g} \rangle$.

The unit map of this free--forgetful adjunction sends $f \in M$ to the generator $\overline{f} \in \MonIdl M$, which corresponds to the monoid ideal $fM$.
Note that this map is not injective in general and so we may gain some further understanding by replacing the generating set with the relevant quotient of $M$.

The order structure on $\MonIdl M$ induces an preorder on $M$ which is the opposite of the usual \emph{divisibility preorder}: $fM \subseteq gM \iff f \in gM \iff \exists k \in M.\ f = gk \iff g \mid f$.
The equivalence relation induced by this preorder is a monoid congruence and so we may quotient $M$ by it to obtain a monoid which is partially ordered by (the reverse of) divisibility.
Such a monoid is sometimes said to be \emph{naturally partially ordered} or a \emph{holoid}.

We can express this holoid quotient $M/{\sim}$ as the following coinserter in the 2-category of posets.
\begin{center}
  \begin{tikzpicture}[node distance=2.5cm, auto]
    \node (A) {$M \times M$};
    \node (B) [right of=A] {$M$};
    \node (C) [right of=B] {$M/{\sim}$};
    \draw[transform canvas={yshift=0.6ex},->] (A) to node {$\pi_1$} (B);
    \draw[transform canvas={yshift=-0.6ex},->] (A) to [swap] node {$\mu_\times$} (B);
    \draw[->>] (B) to node {} (C);
  \end{tikzpicture}
\end{center}

For a commutative localic monoid $M$, we can then take the same coinserter in $\Loc$ obtaining the following inserter in $\Frm$.
\begin{center}
  \begin{tikzpicture}[node distance=2.5cm, auto]
    \node (A) {$\Sats M$};
    \node (B) [right of=A] {$\O M$};
    \node (C) [right of=B] {$\O M \oplus \O M$};
    \draw[transform canvas={yshift=0.6ex},->] (B) to node {$\iota_1$} (C);
    \draw[transform canvas={yshift=-0.6ex},->] (B) to [swap] node {$\mu_\times$} (C);
    \draw[right hook->] (A) to node {$\kappa$} (B);
  \end{tikzpicture}
\end{center}

The resulting frame $\Sats M$ consists of what we call the \emph{saturated opens} of $M$. These are the elements $s \in \O M$ for which $\mu_\times(s) \le \iota_1(s)$ ---
or equivalently, for which $xy \in s \vdash_{x,y} x \in s$ in the internal logic.
For a discrete ring, these are precisely the \emph{saturated sets} of commutative algebra.

\begin{lemma}\label{prop:homomorphism_preserve_saturated_elt}
 Comonoid homomorphisms in $\Frm$ preserve saturated opens.
\end{lemma}
\begin{proof}
 Suppose $f\colon \O M \to \O M'$ is a morphism of cocommutative comonoids and $s$ is a saturated element in $\O M$.
 Then $\mu'_\times f (s) = (f \oplus f) \mu_\times (s) \le (f \oplus f) \iota_1^M (s) = \iota_1^{M'} f(s)$ and hence $f(s)$ is saturated, as required.
\end{proof}

\begin{corollary}
 We obtain a functor $\Sats\colon \CComon(\Frm) \to \Frm$ taking a cocommutative comonoid in $\Frm$ to its subframe of saturated elements.
 The inclusion $\kappa$ gives a natural transformation from $\Sats$ to the forgetful functor from $\CComon(\Frm)$ to $\Frm$.
\end{corollary}

We can say more about $\Sats M$ when $M$ is overt. In that case $\iota_1$ has a left adjoint so that $\mu_\times(s) \le \iota_1(s) \iff (\iota_1)_!\mu_\times(s) \le s$.
\begin{lemma}
 Let $M$ be an overt commutative localic monoid. Then $(\iota_1)_!\mu_\times$ is a closure operator whose fixed points are the saturated elements.
\end{lemma}
\begin{proof}
 It is clearly monotone. To show it is inflationary we need $a \le (\iota_1)_! \mu_\times(a)$, which in the internal logic means $x \in a \vdash_x \exists y\colon M.\ xy \in a$.
 But this holds for $y = 1$ (the unit of the monoid), since $x \mult 1 =_{x} x$ by the unit axiom.
 
 To prove idempotence we require $(\iota_1)_! \mu_\times (\iota_1)_! \mu_\times(a) \le (\iota_1)_! \mu_\times(a)$. Let us translate this into the internal logic.
 The open $(\iota_1)_! \mu_\times(a)$ can be described by the formula $\exists y'.\ x'y' \in a$ in the context $x' \colon M$.
 Then $\mu_\times (\iota_1)_! \mu_\times(a)$ corresponds to substituting $xz$ for $x'$ to give the formula $\exists y'.\ (xz)y' \in a$ in the context $x\colon M, z\colon M$.
 Finally, $(\iota_1)_! \mu_\times (\iota_1)_! \mu_\times(a)$ is obtained by existentially quantifying over $z$ to give the formula $\exists z.\ \exists y'.\ (xz)y' \in a$.
 Thus, we must prove the sequent $\exists z.\ \exists y'.\ (xz)y' \in a \vdash_x \exists y.\ xy \in a$. This is proved by taking $y = zy'$ and using associativity.
 
 The fixed points are the elements such that $(\iota_1)_!\mu_\times(a) \le a$. But these are precisely the saturated elements.
\end{proof}

Note that this implies that for overt $M$ the inclusion $\kappa$ has a left adjoint, which we write as $\kappa_!$ and obtain from $(\iota_1)_!\mu_\times$ by restricting its codomain to its image.

\begin{corollary}
 Under these assumptions, the inserter
 \begin{center}
  \begin{tikzpicture}[node distance=2.5cm, auto]
    \node (A) {$\Sats M$};
    \node (B) [right of=A] {$\O M$};
    \node (C) [right of=B] {$\O M \otimes \O M$};
    \draw[transform canvas={yshift=0.6ex},->] (B) to node {$\iota_1$} (C);
    \draw[transform canvas={yshift=-0.6ex},->] (B) to [swap] node {$\mu_\times$} (C);
    \draw[right hook->] (A) to node {$\kappa$} (B);
  \end{tikzpicture}
\end{center}
is an absolute weighted limit in $\Sup$.
\end{corollary}
\begin{proof}
 Simply note that $\Sats M$ is obtained by splitting the idempotent $(\iota_1)_!\mu_\times$.
\end{proof}

\begin{lemma}
 If $M$ is an overt commutative localic monoid, then $\Sats M$ is a sub-comonoid of $\O M$.
\end{lemma}
\begin{proof}
 The coproduct $\O M \oplus \O M$ has a natural comonoid structure and the saturation closure operator is given by
 $(\iota_1 \otimes \iota_1)_! (\mu_\times \otimes \mu_\times) = ((\iota_1)_! \otimes (\iota_1)_!) (\mu_\times \otimes \mu_\times) = (\iota_1)_!\mu_\times \otimes (\iota_1)_!\mu_\times =
 \kappa\kappa_! \otimes \kappa\kappa_! = (\kappa \otimes \kappa)(\kappa \otimes \kappa)_!$.
 Therefore, $\kappa \oplus \kappa\colon \Sats M \oplus \Sats M \to \O M \oplus \O M$ is isomorphic to the inclusion of the frame of saturated opens
 of $\O M \oplus \O M$.
 
 Now since $M$ is commutative, the comultiplication map $\mu_\times\colon \O M \to \O M \oplus \O M$ is a comonoid homomorphism and hence preserves saturated elements by \cref{prop:homomorphism_preserve_saturated_elt}.
 Thus, $\mu_\times$ restricts to give a map $\mu^\Sats_\times = \Sats \mu_\times$ shown in the diagram below.
 
 \begin{center}
 \begin{tikzpicture}[node distance=2.85cm, auto]
  \node (A) {$\O M$};
  \node (B) [right of=A] {$\O M \oplus \O M$};
  \node (C) [below of=A] {$\Sats M$};
  \node (D) [right of=C] {$\Sats M \oplus \Sats M$};
  \draw[->] (A) to node {$\mu_\times$} (B);
  \draw[right hook->] (C) to node {$\kappa$} (A);
  \draw[right hook->] (D) to [swap] node {$\kappa \oplus \kappa$} (B);
  \draw[->] (C) to [swap] node {$\mu^\Sats_\times$} (D);
 \end{tikzpicture}
 \end{center}
 
 The (co)associativity of $\mu^\Sats_\times$ follows from that of $\mu_\times$ and the fact that $\kappa^{\oplus 3}$ is monic.
 The counit of $\O M$ gives a counit for $\Sats M$ in the obvious way and we see that $\kappa$ is a comonoid homomorphism.
\end{proof}

\begin{corollary}
 The functor $\Sats$ on the category of overt cocommutative comonoids in $\Frm$ lifts to a functor $\Sats$ from $\CComon(\OFrm)$ to itself and $\kappa$ becomes a natural transformation from $\Sats$ to the identity functor.
\end{corollary}

\begin{definition}
 We say a commutative localic monoid $M$ is \emph{deflationary} if $\mu_\times \le \iota_1$.
\end{definition}
We call such a localic monoid deflationary, since it gives a pointfree description of the condition that for all $x \in M$, the map $(-) \mult x$ is deflationary with respect to the specialisation order.
Note that in a deflationary localic monoid all opens are saturated.

\begin{proposition}
 Overt deflationary comonoids in $\Frm$ form a coreflective subcategory of overt cocommutative comonoids where $\Sats$ is the coreflector and $\kappa$ is the counit.
\end{proposition}
\begin{proof}
 We first show that $\Sats M$ is indeed deflationary.
 To see this, observe that $(\kappa \oplus \kappa) \mu^\Sats_\times = \mu_\times \kappa \le \iota_1 \kappa = (\kappa \oplus \kappa) \iota_1^{\Sats M}$.
 But $\kappa \oplus \kappa$ is an injective frame homomorphism and thus reflects order, so that $\mu^\Sats_\times \le \iota_1^{\Sats M}$ as required.
 
 From this we see that $\Sats \kappa$ is always the identity. But $\kappa_D$ is also the identity for deflationary localic monoids $D$.
 These give the two triangle identities and so the result follows.
\end{proof}

We observed above that the holoid quotient of a discrete monoid $M$ can be used to present the spectrum of $M$.
We can now show a similar result for any overt commutative localic monoid.
\begin{proposition}\label{prop:OPMAI_M_cong_OPMAI_SM}
 If $M$ is an overt commutative localic monoid, then $\overline{\OPMAI}_M \cong \overline{\OPMAI}_{\Sats M}$.
\end{proposition}
\begin{proof}
 The inclusion $\kappa\colon \Sats M \hookrightarrow \O M$ induces a natural transformation from $\overline{\OPMAI}_{\Sats M}$ to $\overline{\OPMAI}_M$ sending $u \in \overline{\OPMAI}_M(Q)$
 to $(Q \oplus \kappa)(u)$.
 Since the inserter defining $\Sats M$ is preserved by $Q \oplus (-)$, the map $Q \oplus \kappa$ represents the inclusion into $Q \oplus \O M$ of the elements $u$ satisfying
 $(Q \oplus \mu_\times)(u) \le (Q \oplus\iota_1)(u)$.
 But every $u \in \overline{\OPMAI}_M(Q)$ satisfies $(Q \oplus \mu_\times)(u) = (Q \oplus\iota_1)(u) \mult (Q \oplus\iota_2)(u) \le (Q \oplus\iota_1)(u)$ and hence this inclusion is the identity.
\end{proof}

We can use this to show the \emph{localic} spectrum exists in a large number of cases.
The following result applies in particular whenever $M$ is overt and locally compact.
\begin{corollary}\label{prop:when_localic_monoid_spectrum_exists}
 Let $M$ be an overt commutative localic monoid. Then the localic monoid spectrum of $M$ exists whenever $\Sats M$ is locally compact.
\end{corollary}
\begin{proof}
 By \cref{prop:OPMAI_M_cong_OPMAI_SM} we may replace $M$ with $\Sats M$ and
 the exponentials used in the generalised presentation for the spectrum of $\Sats M$ exist.
\end{proof}

On the other hand, we cannot ensure the \emph{quantic} spectrum exists unless $\Sats M$ is exponentiable in $\QuantGeneral\op$ --- that is, if it is dualisable (see \cite{Niefield2016}).
We would also like to gain a better understanding of the precise form of the spectrum.

Recall that in the discrete case, the representing object of $\OPMAI_M$ is given by the quantale of monoid ideals.
We now describe how this quantale might be defined more generally.

If $M$ is a localic monoid, then the set of overt weakly closed sublocales of $M$ has a natural quantale structure, since $\hom(\O(-), \Omega)\colon \Loc \to \Sup$ is a lax monoidal functor.
In the discrete case, this corresponds to the elementwise multiplication of subsets.
Explicitly, the unit is given by $\epsilon_0$ and the product of $f,g \colon \O M \to \Omega$ is given by $\Delta (f \otimes g) \mu_\times$, where $\Delta\colon \Omega \otimes \Omega \xrightarrow{\sim} \Omega$ is the codiagonal map.

\begin{definition}
 We call an overt weakly closed sublocale $I$ of a commutative localic monoid $M$ a \emph{monoid ideal} if $I = I \top$.
\end{definition}

In the discrete case this recovers the usual definition: a subset $I$ such that $x \in I, y \in M \implies xy \in I$.
More generally, these are precisely the fixed points of the nucleus $W \mapsto W \top$ associated to the two-sided reflection of $\hom(\O M, \Omega)$.
We call this two-sided reflection the \emph{quantale of monoid ideals} and denote it by $\MonIdl M$.

The following proposition gives a link between $\MonIdl M$ and $\Sats M$.

\begin{proposition}\label{prop:dual_of_SatsR}
 If $M$ is an overt commutative localic monoid, then we have $\MonIdl M \cong \hom(\Sats M, \Omega)$.
\end{proposition}
\begin{proof}
 If $M$ is overt, $\Sats M$ can be obtained by splitting the idempotent $(\iota_1)_! \mu_\times$ in $\Sup$.
 Applying the functor $\hom(-, \Omega)$, we find that $\hom(\Sats M, \Omega)$ is given by splitting the idempotent $\hom((\iota_1)_! \mu_\times, \Omega)$.
 This maps sends $f$ to $f (\iota_1)_! \mu_\times = \Delta (f \otimes \exists) \mu_\times$.
 But this is just the nucleus $W \mapsto W \top$ expressed in terms of suplattice homomorphisms.
\end{proof}

\begin{corollary}
 If $\Sats M$ is dualisable, then $\MonIdl M$ is its dual.
\end{corollary}

Note that the isomorphism $\MonIdl M \cong \hom(\Sats M, \Omega)$ preserves the quantale structure, since the comonoid structure on $\Sats M$ inherited as a subobject of $\O M$
and the quantale structure on $\MonIdl M$ inherited as a quotient of $\hom(\O M, \Omega)$ are both unique.

When $\Sats M$ is dualisable, we can visualise the multiplicative structure on $(\Sats M)^* \cong \MonIdl M$ using string diagrams.
We will assume $M$ is deflationary for simplicity so that $\Sats M \cong \O M$.
Then the monoid structure on $(\O M)^*$ can be expressed as follows.
(We will suppress the $\O$ in string diagrams to avoid clutter and since it is clear we are working in $\Sup$.)
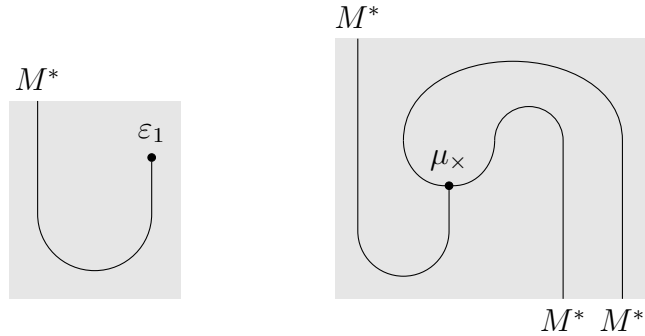
\begin{center}
\vspace{-3pt}
\begin{minipage}[t][][t]{0.27\textwidth}
\centering
\begin{tikzpicture}[scale=0.75]
\path coordinate (eta) ++(1,1) coordinate (a) ++(0,1) coordinate[dot, label=above:$\epsilon_1$] (e)
 (eta) ++(-1,1) coordinate (c) ++(0,2) coordinate[label=above:$M^*$] (tl)
 (eta) ++(1,-0.5) coordinate[label=below:$\phantom{M}$] (b);
\draw (tl) -- (c) to[out=-90, in=180] (eta) to[out=0, in=-90] (a) -- (e);
\fillbackground{$(b) + (0.5,0)$}{$(tl) + (-0.5,0)$};
\end{tikzpicture}
\end{minipage}
\begin{minipage}[t][][t]{0.37\textwidth}
\centering
\begin{tikzpicture}[scale=0.6]
\path coordinate[dot, label=above:$\mu_\times$] (mu)
 +(0,-1) coordinate (a)
 +(-1,1) coordinate (mtl)
 +(1,1) coordinate (mtr);
\path (a) ++(-1.0,-1) coordinate (eta) ++(-1.0,1) coordinate (c) ++(0,4.25) coordinate[label=above:$M^*$] (tl);
\path (mtr) ++(0.75,0.75) coordinate (epsilon) ++(0.75,-0.75) coordinate (d) ++(0,-3.5) coordinate[label=below:$M^*$] (br);
\path (mtl) ++(2.4,1.75) coordinate (epsilon2) ++(2.4,-1.75) coordinate (d2) ++(0,-3.5) coordinate[label=below:$M^*$] (br2);
\path (eta) ++(1.0,-0.5) coordinate (b);
\draw (mtl) to[out=270, in=180] (mu.west) -- (mu.east) to[out=0, in=270] (mtr)
      (mu) -- (a);
\draw (tl) -- (c) to[out=-90, in=180] (eta) to[out=0, in=-90] (a);
\draw (mtr) to[out=90, in=180] (epsilon) to[out=0, in=90] (d) -- (br);
\draw (mtl) to[out=90, in=180] (epsilon2) to[out=0, in=90] (d2) -- (br2);
\fillbackground{$(tl) + (-0.5,0)$}{$(br2) + (0.5,0)$};
\end{tikzpicture}
\end{minipage}
\end{center}

We are now in a position to prove our main result about the monoid spectrum.
\begin{theorem}\label{prop:spectrum_for_localic_monoids}
 Let $M$ be an overt commutative localic monoid and suppose $\Sats M$ is a dualisable suplattice.
 Then $\overline{\OPMAI}_M$ is representable with representing object $(\Sats M)^* \cong \MonIdl M$ and universal element $(\MonIdl M \otimes \kappa)\eta(\top)$,
 where $\eta$ is the unit of the duality between $\Sats M$ and $\MonIdl M$.
\end{theorem}

\begin{proof}
 By \cref{prop:OPMAI_M_cong_OPMAI_SM} we may assume $M$ is deflationary so that $\Sats M \cong \O M$.
 By duality and the universal property of $\Omega$, we have that $\abs{Q \otimes \O M} \cong \Hom_\Sup(\Omega, Q \otimes \O M) \cong \Hom_\Sup(\Omega, \O M \otimes Q) \cong \Hom_\Sup(\Omega \otimes \O M^*, Q) \cong \Hom_\Sup(\O M^*, Q)$.
 Moreover, the identity map $\id_{\O M^*} \in \Hom_\Sup(\O M^*, \O M^*)$ corresponds to $\eta(\top)$ under this bijection.
 So to prove the desired result it is enough to show that if $(Q, m, e)$ is a two-sided quantale, the elements $u$ of $Q \otimes \O M$ satisfying
 $(Q \otimes \mu_\times)(u) = (Q \otimes\iota_1)(u) \mult (Q \otimes\iota_2)(u)$ and $(Q \otimes \epsilon_1)(u) = 1$ correspond to the suplattice homomorphisms from $\O M^*$ to $Q$
 which respect the multiplicative structure.
 
 Consider a suplattice map $f\colon \O M^* \to Q$. The corresponding map $f^\sharp \in \Hom(\Omega, Q \otimes \O M)$ is shown in the following string diagram.
 
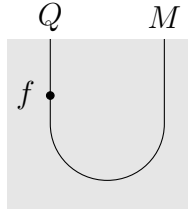
\begin{center}
 \vspace{-3pt}
 \begin{tikzpicture}[scale=0.75]
  \path coordinate (eta) ++(1,1) coordinate (a) ++(0,1.5) coordinate[label=above:$M$] (tr)
  (eta) ++(-1,1) coordinate (c) ++ (0,0.5) coordinate[dot,label=left:$f$] (f) ++(0,1) coordinate[label={[text depth=0.2ex] above:$Q$}] (tl)
  (eta) ++(1,-0.5) coordinate (b);
  \draw (tl) -- (c) to[out=-90, in=180] (eta) to[out=0, in=-90] (a) -- (tr);
  \fillbackground{$(b) + (0.5,0)$}{$(tl) + (-0.75,0)$};
 \end{tikzpicture}
 \end{center}
 
 Now we have $(Q \otimes \epsilon_1)f^\sharp (\top) = 1$ if
 \\ 
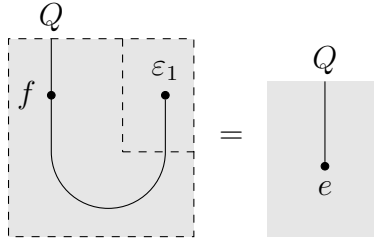
\vspace{-3pt}
 \begin{equation*}
 \begin{tikzpicture}[scale=0.75,baseline={([yshift=-0.5ex]current bounding box.center)}]
  \path coordinate (eta) ++(1,1) coordinate (a) ++(0,1) coordinate[dot, label=above:$\epsilon_1$] (e)
  (eta) ++(-1,1) coordinate (c) ++ (0,1) coordinate[dot,label=left:$f$] (f) ++(0,1) coordinate[label={[text depth=0.2ex] above:$Q$}] (tl)
  (eta) ++(1,-0.5) coordinate[label=below:$\phantom{M}$] (b);
  \draw (tl) -- (c) to[out=-90, in=180] (eta) to[out=0, in=-90] (a) -- (e);
  \draw[dashed] ($(e)-(0.75,1)$) rectangle ($(e)+(0.5,1)$);
  \draw[dashed] ($(e)+(-0.75,1)$) -| ($(f)-(0.75,1)$) |- ($(b) + (0.5,0)$) -- ($(e)+(0.5,-1)$);
  \fillbackground{$(b) + (0.5,0)$}{$(tl) + (-0.75,0)$};
 \end{tikzpicture}
 \enspace=\enspace
 \begin{tikzpicture}[scale=0.75,baseline={([yshift=-0.5ex]current bounding box.center)}]
  \path coordinate[dot, label=below:$e$] (e) ++(0,1.5) coordinate[label={[text depth=0.2ex] above:$Q$}] (t);
  \draw (e) -- (t);
  \fillbackground{$(t) + (-1,0)$}{$(e) + (1,-1.25)$};
 \end{tikzpicture}
 \end{equation*}
 while $f$ preserves the multiplicative unit if
 \\ 
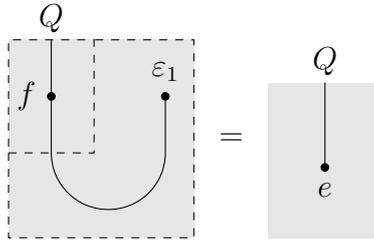
\vspace{-3pt}
 \begin{equation*}
 \begin{tikzpicture}[scale=0.75,baseline={([yshift=-0.5ex]current bounding box.center)}]
  \path coordinate (eta) ++(1,1) coordinate (a) ++(0,1) coordinate[dot, label=above:$\epsilon_1$] (e)
  (eta) ++(-1,1) coordinate (c) ++ (0,1) coordinate[dot,label=left:$f$] (f) ++(0,1) coordinate[label={[text depth=0.2ex] above:$Q$}] (tl)
  (eta) ++(1,-0.5) coordinate[label=below:$\phantom{M}$] (b);
  \draw (tl) -- (c) to[out=-90, in=180] (eta) to[out=0, in=-90] (a) -- (e);
  \draw[dashed] ($(f)-(0.75,1)$) rectangle ($(f)+(0.75,1)$);
  \draw[dashed] ($(f)-(0.75,1)$) |- ($(b) + (0.5,0)$) |- ($(f)+(0.75,1)$);
  \fillbackground{$(b) + (0.5,0)$}{$(tl) + (-0.75,0)$};
 \end{tikzpicture}
 \enspace=\enspace
 \begin{tikzpicture}[scale=0.75,baseline={([yshift=-0.5ex]current bounding box.center)}]
  \path coordinate[dot, label=below:$e$] (e) ++(0,1.5) coordinate[label={[text depth=0.2ex] above:$Q$}] (t);
  \draw (e) -- (t);
  \fillbackground{$(t) + (-1,0)$}{$(e) + (1,-1.25)$};
 \end{tikzpicture}
 \end{equation*}
 These are the same two diagrams and hence the two conditions are clearly equivalent.
 
 Now consider the condition $(Q \otimes \mu_\times) f^\sharp (\top) = (Q \otimes\iota_1) f^\sharp (\top) \mult (Q \otimes\iota_2) f^\sharp (\top)$.
 Let us start with the right-hand side.
 It can be rewritten as $(m \otimes \O M \otimes \O M)(Q \otimes \sigma_{\O M,Q} \otimes \O M)(f^\sharp \otimes f^\sharp) (\top)$ where $\sigma_{\O M,Q}\colon \O M \otimes Q \cong Q \otimes \O M$ is the symmetry map.
 This is depicted in the first string diagram below.
 This can then be manipulated to give the second diagram
 and we obtain the third diagram using the commutativity of $m$.
 \\ 
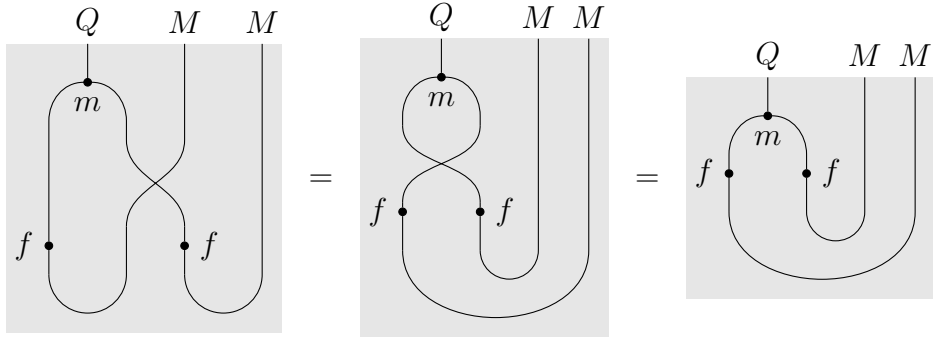
\vspace{-3pt}
 \begin{equation*}
 \begin{tikzpicture}[scale=0.51,baseline={([yshift=-1.5ex]current bounding box.center)}]
  \path coordinate[dot, label=below:$m$] (mu)
   +(0,1) coordinate[label={[text depth=0.2ex] above:$Q$}] (t)
   +(-1.0,-1.0) coordinate (mbl)
   +(1.0,-1.0) coordinate (mbr);
  \path (mbl) ++ (0,-3.25) coordinate[dot,label=left:$f$] (fl) ++ (0,-0.75) coordinate (a1);
  \path (mbr) ++ (0,-0.5) coordinate (el) ++ (1.5,-2.25) coordinate (dr) ++ (0,-0.5) coordinate[dot,label=right:$f$] (fr) ++ (0,-0.75) coordinate (a2);
  \path (a2) ++(1.0,-1.0) coordinate (eta2) ++(1.0,1.0) coordinate (c2) ++ (0,6.0) coordinate[label=above:$M$] (tr2);
  \path (a1) ++(1.0,-1.0) coordinate (eta1) ++(1.0,1.0) coordinate (c1) ++ (0,1.25) coordinate (dl) ++ (1.5,2.25) coordinate (er) ++ (0,2.5) coordinate[label=above:$M$] (tr1);
  \draw (mbl) to[out=90, in=180] (mu.west) -- (mu.east) to[out=0, in=90] (mbr)
        (mu) -- (t)
        (mbl) -- (fl) -- (a1)
        (mbr) -- (el)
        (dr) -- (fr) -- (a2);
  \draw (er) to[out=270, in=90] (dl)
        (el) to[out=270, in=90] (dr);
  \draw (tr1) -- (er)
        (dl) -- (c1) to[out=-90, in=0] (eta1) to[out=180, in=-90] (a1);
  \draw (tr2) -- (c2) to[out=-90, in=0] (eta2) to[out=180, in=-90] (a2);
  \fillbackground{$(eta1) + (-2.1,-0.5)$}{$(tr2) + (0.5,0)$};
 \end{tikzpicture}
 \enspace\!=\enspace
 \begin{tikzpicture}[scale=0.51,baseline={([yshift=-1.5ex]current bounding box.center)}]
  \path coordinate[dot, label=below:$m$] (mu)
   +(0,1) coordinate[label={[text depth=0.2ex] above:$Q$}] (t)
   +(-1,-1) coordinate (mbl)
   +(1,-1) coordinate (mbr);
  \path (mbl) ++ (0,-0.25) coordinate (el) ++ (0,-2.0) coordinate (dl) ++ (0,-0.25) coordinate[dot,label=left:$f$] (fl) ++ (0,-1) coordinate (a1)
        (mbr) ++ (0,-0.25) coordinate (er) ++ (0,-2.0) coordinate (dr) ++ (0,-0.25) coordinate[dot,label=right:$f$] (fr) ++ (0,-1) coordinate (a2);
  \path (a2) ++(0.75,-0.75) coordinate (eta2) ++(0.75,0.75) coordinate (c2) ++ (0,5.5) coordinate[label=above:$M$] (tr2);
  \path (a1) ++(2.4,-1.75) coordinate (eta1) ++(2.4,1.75) coordinate (c1) ++ (0,5.5) coordinate[label=above:$M$] (tr1);
  \draw (mbl) to[out=90, in=180] (mu.west) -- (mu.east) to[out=0, in=90] (mbr)
        (mbl) -- (el)
        (dl) -- (fl) -- (a1)
        (mbr) -- (er)
        (dr) -- (fr) -- (a2)
        (mu) -- (t);
  \draw (er) to[out=270, in=90] (dl)
        (el) to[out=270, in=90] (dr);
  \draw (tr2) -- (c2) to[out=-90, in=0] (eta2) to[out=180, in=-90] (a2);
  \draw (tr1) -- (c1) to[out=-90, in=0] (eta1) to[out=180, in=-90] (a1);
  \fillbackground{$(eta1) + (-3.5,-0.5)$}{$(tr1) + (0.5,0)$};
 \end{tikzpicture}
 \!\enspace=\enspace
 \begin{tikzpicture}[scale=0.51,baseline={([yshift=-1.5ex]current bounding box.center)}]
  \path coordinate[dot, label=below:$m$] (mu)
   +(0,1) coordinate[label={[text depth=0.2ex] above:$Q$}] (t)
   +(-1,-1) coordinate (mbl)
   +(1,-1) coordinate (mbr);
  \path (mbl) ++ (0,-0.5) coordinate[dot,label=left:$f$] (fl) ++ (0,-1) coordinate (a1)
        (mbr) ++ (0,-0.5) coordinate[dot,label=right:$f$] (fr) ++ (0,-1) coordinate (a2);
  \path (a2) ++(0.75,-0.75) coordinate (eta2) ++(0.75,0.75) coordinate (c2) ++ (0,3.5) coordinate[label=above:$M$] (tr2);
  \path (a1) ++(2.4,-1.75) coordinate (eta1) ++(2.4,1.75) coordinate (c1) ++ (0,3.5) coordinate[label=above:$M$] (tr1);
  \draw (mbl) to[out=90, in=180] (mu.west) -- (mu.east) to[out=0, in=90] (mbr)
        (mbl) -- (fl) -- (a1)
        (mbr) -- (fr) -- (a2)
        (mu) -- (t);
  \draw (tr2) -- (c2) to[out=-90, in=0] (eta2) to[out=180, in=-90] (a2);
  \draw (tr1) -- (c1) to[out=-90, in=0] (eta1) to[out=180, in=-90] (a1);
  \fillbackground{$(eta1) + (-3.5,-0.5)$}{$(tr1) + (0.5,0)$};
 \end{tikzpicture}
 \end{equation*}
 
 Combining this with the left-hand side, we have that the condition can be represented as follows.
 \\ 
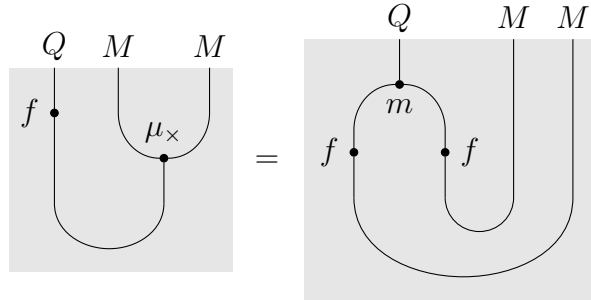
\vspace{-3pt}
 \begin{equation*}
 \begin{tikzpicture}[scale=0.6,baseline={([yshift=-1.5ex]current bounding box.center)}]
  \path coordinate[dot, label=above:$\mu_\times$] (mu)
  +(0,-1) coordinate (a)
  +(-1,1) coordinate (mtl)
  +(1,1) coordinate (mtr);
  \path (mtl) ++(0,1) coordinate[label=above:$M$] (tl)
        (mtr) ++(0,1) coordinate[label=above:$M$] (tr);
  \path (a) ++(-1.2,-1) coordinate (eta) ++(-1.2,1) coordinate (c) ++(0,2) coordinate[dot,label=left:$f$] (f) ++ (0,1) coordinate[label={[text depth=0.2ex] above:$Q$}] (tll);
  \path (eta) ++(1.2,-0.5) coordinate (b);
  \draw (tl) -- (mtl) to[out=270, in=180] (mu.west) -- (mu.east) to[out=0, in=270] (mtr) -- (tr)
        (mu) -- (a);
  \draw (tll) -- (c) to[out=-90, in=180] (eta) to[out=0, in=-90] (a);
  \fillbackground{$(tll) + (-1.0,0)$}{$(b) + (1.5,0)$};
 \end{tikzpicture}
 \;=\enspace
 \begin{tikzpicture}[scale=0.6,baseline={([yshift=-1.5ex]current bounding box.center)}]
  \path coordinate[dot, label=below:$m$] (mu)
   +(0,1) coordinate[label={[text depth=0.2ex] above:$Q$}] (t)
   +(-1,-1) coordinate (mbl)
   +(1,-1) coordinate (mbr);
  \path (mbl) ++ (0,-0.5) coordinate[dot,label=left:$f$] (fl) ++ (0,-1) coordinate (a1)
        (mbr) ++ (0,-0.5) coordinate[dot,label=right:$f$] (fr) ++ (0,-1) coordinate (a2);
  \path (a2) ++(0.75,-0.75) coordinate (eta2) ++(0.75,0.75) coordinate (c2) ++ (0,3.5) coordinate[label=above:$M$] (tr2);
  \path (a1) ++(2.4,-1.75) coordinate (eta1) ++(2.4,1.75) coordinate (c1) ++ (0,3.5) coordinate[label=above:$M$] (tr1);
  \draw (mbl) to[out=90, in=180] (mu.west) -- (mu.east) to[out=0, in=90] (mbr)
        (mbl) -- (fl) -- (a1)
        (mbr) -- (fr) -- (a2)
        (mu) -- (t);
  \draw (tr2) -- (c2) to[out=-90, in=0] (eta2) to[out=180, in=-90] (a2);
  \draw (tr1) -- (c1) to[out=-90, in=0] (eta1) to[out=180, in=-90] (a1);
  \fillbackground{$(eta1) + (-3.5,-0.5)$}{$(tr1) + (0.5,0)$};
 \end{tikzpicture}
 \end{equation*}
 
 On the other hand, the map $f$ preserves the multiplication if
 \\ 
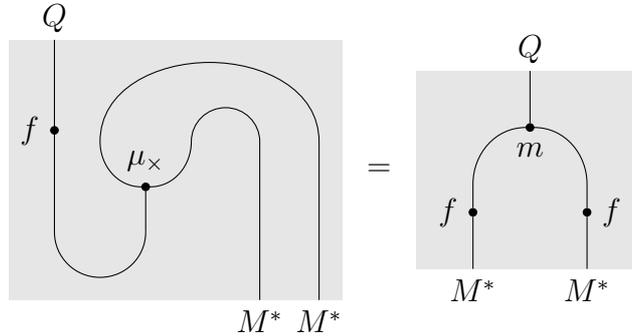
\vspace{-3pt}
 \begin{equation*}
 \begin{tikzpicture}[scale=0.6,baseline={([yshift=-0.5ex]current bounding box.center)}]
  \path coordinate[dot, label=above:$\mu_\times$] (mu)
  +(0,-1) coordinate (a)
  +(-1,1) coordinate (mtl)
  +(1,1) coordinate (mtr);
  \path (a) ++(-1.0,-1) coordinate (eta) ++(-1.0,1) coordinate (c) ++(0,2.25) coordinate[dot,label=left:$f$] (f) ++(0,2) coordinate[label={[text depth=0.2ex] above:$Q$}] (tl);
  \path (mtr) ++(0.75,0.75) coordinate (epsilon) ++(0.75,-0.75) coordinate (d) ++(0,-3.5) coordinate[label=below:$M^*$] (br);
  \path (mtl) ++(2.4,1.75) coordinate (epsilon2) ++(2.4,-1.75) coordinate (d2) ++(0,-3.5) coordinate[label=below:$M^*$] (br2);
  \path (eta) ++(1.0,-0.5) coordinate (b);
  \draw (mtl) to[out=270, in=180] (mu.west) -- (mu.east) to[out=0, in=270] (mtr)
        (mu) -- (a);
  \draw (tl) -- (c) to[out=-90, in=180] (eta) to[out=0, in=-90] (a);
  \draw (mtr) to[out=90, in=180] (epsilon) to[out=0, in=90] (d) -- (br);
  \draw (mtl) to[out=90, in=180] (epsilon2) to[out=0, in=90] (d2) -- (br2);
  \fillbackground{$(tl) + (-1.0,0)$}{$(br2) + (0.5,0)$};
 \end{tikzpicture}
 \,=\enspace
 \begin{tikzpicture}[scale=0.75,baseline={([yshift=-0.5ex]current bounding box.center)}]
  \path coordinate[dot, label=below:$m$] (mu)
   +(0,1) coordinate[label={[text depth=0.2ex] above:$Q$}] (t)
   +(-1,-1) coordinate (mbl)
   +(1,-1) coordinate (mbr);
  \path (mbl) ++ (0,-0.5) coordinate[dot,label=left:$f$] (fl) ++ (0,-1) coordinate[label=below:$M^*$] (bl)
        (mbr) ++ (0,-0.5) coordinate[dot,label=right:$f$] (fr) ++ (0,-1) coordinate[label=below:$M^*$] (br);
  \draw (mbl) to[out=90, in=180] (mu.west) -- (mu.east) to[out=0, in=90] (mbr)
        (mbl) -- (fl) -- (bl)
        (mbr) -- (fr) -- (br)
        (mu) -- (t);
  \fillbackground{$(bl) + (-1.0,0)$}{$(t) + (2.0,0)$};
 \end{tikzpicture}
 \end{equation*}
 
 These two equalities can be turned into each other by `bending the wires' and using the duality identities to `pull them straight'.
 These are inverse operations and so the two conditions are equivalent, as required.
\end{proof}

\section{Semiring spectra}

We would like to prove an extension of \cref{prop:spectrum_for_localic_monoids} for localic \emph{semirings}.
Some complications arise, because the inserter $\kappa\colon \Sats R \to \O R$ need not respect the additive structure of a semiring $R$.
However, $\Sats R$ does have an additive counital comagma structure in $\Sup$ given by $\widetilde{\mu}_+ = (\kappa_! \otimes \kappa_!) \mu_+\kappa$ and $\widetilde{\epsilon}_0 = \epsilon_0 \kappa$.
This will be sufficient to carry the construction through in \cref{prop:quantic_spectrum_for_localic_semirings}.

Let us start by defining a quantale of (overt, weakly closed) ideals of a localic semiring $R$ so as to give a candidate representing object of $\overline{\OPAI}_R$.
Similarly to the multiplication, the addition on $R$ also induces a quantale structure on $\hom(\O R, \Omega)$, whose binary operation we write as $+$ and whose unit we denote by $0_+$.

\begin{definition}
 We call an overt weakly closed sublocale $I$ of a localic semiring $R$ an \emph{ideal} if it is a monoid ideal, $0_+ \le I$ and $I + I \le I$.
\end{definition}

These are the largest elements of the congruence classes of the quantale congruence on $\hom(\O R, \Omega)$ (with its \emph{multiplicative} structure) obtained by setting $1 = \top$, $0_+ = 0$ and $a + b \le a \vee b$.
The resulting quotient is called the \emph{quantale of ideals} $\Idl(R)$.

The quantale $\Idl(R)$ is two-sided and so its defining quotient factors through $\MonIdl R$ to give a quotient map $\addquot\colon \MonIdl R \twoheadrightarrow \Idl(R)$.
It will be useful to have a description of this quotient in terms of operations on $\MonIdl R$.

Dualising the modified additive structure on $\Sats R$ described above yields a bilinear operation on $\MonIdl R$
given by $I \mathop{\widetilde+} J = (I + J) \top$ with unit $\widetilde{0}_+ = 0_+ \top = 0_+$.
We then easily obtain the following lemma.
\begin{lemma}\label{prop:kernel_congruence_of_addquot}
 The kernel congruence of $\addquot\colon \MonIdl R \twoheadrightarrow \Idl(R)$ is generated by setting $\widetilde{0}_+ = 0$ and $I \mathop{\widetilde+} J \le I \vee J$.
\end{lemma}

We can now prove our main result.

\begin{theorem}\label{prop:quantic_spectrum_for_localic_semirings}
 Let $R$ be an overt commutative localic semiring and suppose $\Sats R$ is a dualisable suplattice.
 Then $\overline{\OPAI}_R$ is representable with representing object $\Idl(R)$ and universal element $(\addquot \otimes \kappa)\eta(\top)$,
 where $\eta$ is the unit of the duality between $\Sats R$ and $\MonIdl R$.
\end{theorem}
\begin{proof}
 \Cref{prop:spectrum_for_localic_monoids} gives an isomorphism $\Hom_\Quant(\MonIdl R, -) \cong \OPMAI_R$.
 We have that $\overline{\OPAI}_R(Q)$ is a subset of $\overline{\OPMAI}_R(Q)$
 and we wish to find which of the homomorphisms in $\Hom_\Quant(\MonIdl R, Q)$ correspond to this subset.
 
 Let $f\colon \MonIdl R \to Q$ be a homomorphism of two-sided quantales.
 For the corresponding (quantic) open prime monoid anti-ideal $(Q \otimes \kappa)f^\sharp(\top) \in Q \otimes \O R$ to be a semiring anti-ideal, we require $(Q \otimes \epsilon_0 \kappa)f^\sharp(\top) = 0$ and
 $(Q \otimes \mu_+ \kappa)f^\sharp(\top) \le (Q \otimes \iota_1 \kappa)f^\sharp(\top) \vee (Q \otimes \iota_2 \kappa)f^\sharp(\top)$.
 
 Similarly to the condition on the multiplicative unit in \cref{prop:spectrum_for_localic_monoids}, we find the first of these is equivalent to $f(\widetilde{0}_+) = 0$.
 For the second condition, observe that $(Q \otimes \iota_1 \kappa)f^\sharp(\top) \vee (Q \otimes \iota_2 \kappa)f^\sharp(\top) =
 (Q \otimes \kappa \otimes \kappa)\left( (Q \otimes \iota_1)f^\sharp(\top) \vee (Q \otimes \iota_2)f^\sharp(\top) \right)$.
 Using the adjunction then yields the equivalent condition
 $(Q \otimes \widetilde{\mu}_+)f^\sharp(\top) \le (Q \otimes \iota_1)f^\sharp(\top) \vee (Q \otimes \iota_2)f^\sharp(\top)$,
 which is depicted in the following string diagrams in $\Sup$.
  \\ 
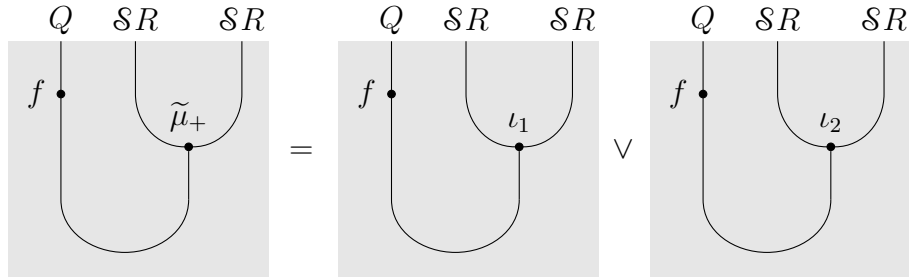
\vspace{-3pt}
 \begin{equation*}
 \begin{tikzpicture}[scale=0.7,baseline={([yshift=-1.5ex]current bounding box.center)}]
  \path coordinate[dot, label=above:$\widetilde{\mu}_+$] (mu)
  +(0,-1) coordinate (a)
  +(-1,1) coordinate (mtl)
  +(1,1) coordinate (mtr);
  \path (mtl) ++(0,1) coordinate[label=above:$\Sats R$] (tl)
        (mtr) ++(0,1) coordinate[label=above:$\Sats R$] (tr);
  \path (a) ++(-1.2,-1) coordinate (eta) ++(-1.2,1) coordinate (c) ++(0,2) coordinate[dot,label=left:$f$] (f) ++ (0,1) coordinate[label={[text depth=0.2ex] above:$Q$}] (tll);
  \path (eta) ++(1.2,-0.5) coordinate (b);
  \draw (tl) -- (mtl) to[out=270, in=180] (mu.west) -- (mu.east) to[out=0, in=270] (mtr) -- (tr)
        (mu) -- (a);
  \draw (tll) -- (c) to[out=-90, in=180] (eta) to[out=0, in=-90] (a);
  \fillbackground{$(tll) + (-1.0,0)$}{$(b) + (1.5,0)$};
 \end{tikzpicture}
 \,=\enspace
 \begin{tikzpicture}[scale=0.7,baseline={([yshift=-1.5ex]current bounding box.center)}]
  \path coordinate[dot, label=above:$\iota_1$] (mu)
  +(0,-1) coordinate (a)
  +(-1,1) coordinate (mtl)
  +(1,1) coordinate (mtr);
  \path (mtl) ++(0,1) coordinate[label=above:$\Sats R$] (tl)
        (mtr) ++(0,1) coordinate[label=above:$\Sats R$] (tr);
  \path (a) ++(-1.2,-1) coordinate (eta) ++(-1.2,1) coordinate (c) ++(0,2) coordinate[dot,label=left:$f$] (f) ++ (0,1) coordinate[label={[text depth=0.2ex] above:$Q$}] (tll);
  \path (eta) ++(1.2,-0.5) coordinate (b);
  \draw (tl) -- (mtl) to[out=270, in=180] (mu.west) -- (mu.east) to[out=0, in=270] (mtr) -- (tr)
        (mu) -- (a);
  \draw (tll) -- (c) to[out=-90, in=180] (eta) to[out=0, in=-90] (a);
  \fillbackground{$(tll) + (-1.0,0)$}{$(b) + (1.5,0)$};
 \end{tikzpicture}
 \vee\;
 \begin{tikzpicture}[scale=0.7,baseline={([yshift=-1.5ex]current bounding box.center)}]
  \path coordinate[dot, label=above:$\iota_2$] (mu)
  +(0,-1) coordinate (a)
  +(-1,1) coordinate (mtl)
  +(1,1) coordinate (mtr);
  \path (mtl) ++(0,1) coordinate[label=above:$\Sats R$] (tl)
        (mtr) ++(0,1) coordinate[label=above:$\Sats R$] (tr);
  \path (a) ++(-1.2,-1) coordinate (eta) ++(-1.2,1) coordinate (c) ++(0,2) coordinate[dot,label=left:$f$] (f) ++ (0,1) coordinate[label={[text depth=0.2ex] above:$Q$}] (tll);
  \path (eta) ++(1.2,-0.5) coordinate (b);
  \draw (tl) -- (mtl) to[out=270, in=180] (mu.west) -- (mu.east) to[out=0, in=270] (mtr) -- (tr)
        (mu) -- (a);
  \draw (tll) -- (c) to[out=-90, in=180] (eta) to[out=0, in=-90] (a);
  \fillbackground{$(tll) + (-1.0,0)$}{$(b) + (1.5,0)$};
 \end{tikzpicture}
 \end{equation*}
 
 Now bending the wires gives the equivalent condition:
 \\ 
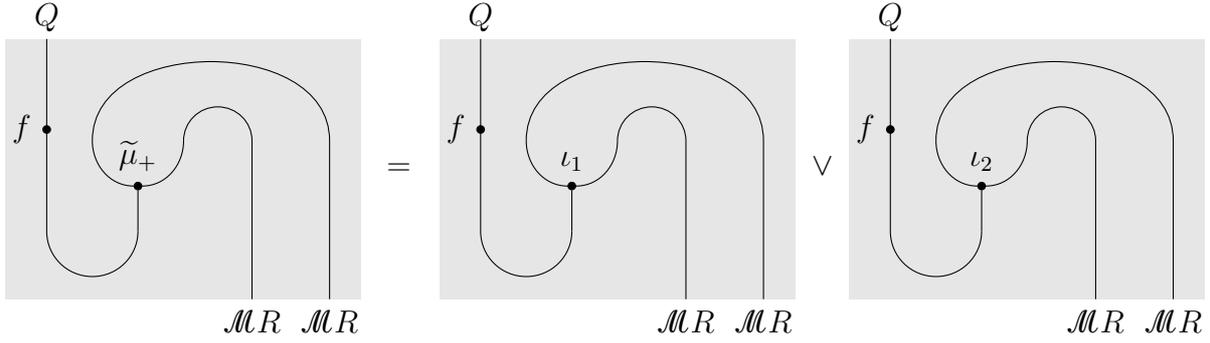
\vspace{-4pt}
 \begin{equation*}
 \begin{tikzpicture}[scale=0.6,baseline={([yshift=-0.5ex]current bounding box.center)}]
  \path coordinate[dot, label=above:$\widetilde{\mu}_+$] (mu)
  +(0,-1) coordinate (a)
  +(-1,1) coordinate (mtl)
  +(1,1) coordinate (mtr);
  \path (a) ++(-1.0,-1) coordinate (eta) ++(-1.0,1) coordinate (c) ++(0,2.25) coordinate[dot,label=left:$f$] (f) ++(0,2) coordinate[label={[text depth=0.2ex] above:$Q$}] (tl);
  \path (mtr) ++(0.75,0.75) coordinate (epsilon) ++(0.75,-0.75) coordinate (d) ++(0,-3.5) coordinate[label=below:$\MonIdl R$] (br);
  \path (mtl) ++(2.6,1.75) coordinate (epsilon2) ++(2.6,-1.75) coordinate (d2) ++(0,-3.5) coordinate[label=below:$\MonIdl R$] (br2);
  \path (eta) ++(1.0,-0.5) coordinate (b);
  \draw (mtl) to[out=270, in=180] (mu.west) -- (mu.east) to[out=0, in=270] (mtr)
        (mu) -- (a);
  \draw (tl) -- (c) to[out=-90, in=180] (eta) to[out=0, in=-90] (a);
  \draw (mtr) to[out=90, in=180] (epsilon) to[out=0, in=90] (d) -- (br);
  \draw (mtl) to[out=90, in=180] (epsilon2) to[out=0, in=90] (d2) -- (br2);
  \fillbackground{$(tl) + (-0.9,0)$}{$(br2) + (0.7,0)$};
 \end{tikzpicture}
 \;=\enspace
 \begin{tikzpicture}[scale=0.6,baseline={([yshift=-0.5ex]current bounding box.center)}]
  \path coordinate[dot, label=above:$\iota_1$] (mu)
  +(0,-1) coordinate (a)
  +(-1,1) coordinate (mtl)
  +(1,1) coordinate (mtr);
  \path (a) ++(-1.0,-1) coordinate (eta) ++(-1.0,1) coordinate (c) ++(0,2.25) coordinate[dot,label=left:$f$] (f) ++(0,2) coordinate[label={[text depth=0.2ex] above:$Q$}] (tl);
  \path (mtr) ++(0.75,0.75) coordinate (epsilon) ++(0.75,-0.75) coordinate (d) ++(0,-3.5) coordinate[label=below:$\MonIdl R$] (br);
  \path (mtl) ++(2.6,1.75) coordinate (epsilon2) ++(2.6,-1.75) coordinate (d2) ++(0,-3.5) coordinate[label=below:$\MonIdl R$] (br2);
  \path (eta) ++(1.0,-0.5) coordinate (b);
  \draw (mtl) to[out=270, in=180] (mu.west) -- (mu.east) to[out=0, in=270] (mtr)
        (mu) -- (a);
  \draw (tl) -- (c) to[out=-90, in=180] (eta) to[out=0, in=-90] (a);
  \draw (mtr) to[out=90, in=180] (epsilon) to[out=0, in=90] (d) -- (br);
  \draw (mtl) to[out=90, in=180] (epsilon2) to[out=0, in=90] (d2) -- (br2);
  \fillbackground{$(tl) + (-0.9,0)$}{$(br2) + (0.7,0)$};
 \end{tikzpicture}
 \,\vee\;
 \begin{tikzpicture}[scale=0.6,baseline={([yshift=-0.5ex]current bounding box.center)}]
  \path coordinate[dot, label=above:$\iota_2$] (mu)
  +(0,-1) coordinate (a)
  +(-1,1) coordinate (mtl)
  +(1,1) coordinate (mtr);
  \path (a) ++(-1.0,-1) coordinate (eta) ++(-1.0,1) coordinate (c) ++(0,2.25) coordinate[dot,label=left:$f$] (f) ++(0,2) coordinate[label={[text depth=0.2ex] above:$Q$}] (tl);
  \path (mtr) ++(0.75,0.75) coordinate (epsilon) ++(0.75,-0.75) coordinate (d) ++(0,-3.5) coordinate[label=below:$\MonIdl R$] (br);
  \path (mtl) ++(2.6,1.75) coordinate (epsilon2) ++(2.6,-1.75) coordinate (d2) ++(0,-3.5) coordinate[label=below:$\MonIdl R$] (br2);
  \path (eta) ++(1.0,-0.5) coordinate (b);
  \draw (mtl) to[out=270, in=180] (mu.west) -- (mu.east) to[out=0, in=270] (mtr)
        (mu) -- (a);
  \draw (tl) -- (c) to[out=-90, in=180] (eta) to[out=0, in=-90] (a);
  \draw (mtr) to[out=90, in=180] (epsilon) to[out=0, in=90] (d) -- (br);
  \draw (mtl) to[out=90, in=180] (epsilon2) to[out=0, in=90] (d2) -- (br2);
  \fillbackground{$(tl) + (-0.9,0)$}{$(br2) + (0.7,0)$};
 \end{tikzpicture}
 \end{equation*}
 The left-hand diagram is a map sending $j \otimes k$ to $f(j \mathop{\widetilde +} k)$. For the right-hand side, viewing $\MonIdl R$ as $\Hom(\Sats R, \Omega)$
 we observe that the mate of $\iota_1\colon \Sats R \to \Sats R \otimes \Sats R$
 sends $j \otimes k$ to the map $x \mapsto (j \otimes k)\iota_1(x) = j(x) \scalarMult \exists(k)$, where $\exists(k) = k(\top)$.
 Thus, this means $f(a \mathop{\widetilde +} b) \le f(a) \scalarMult \exists(b) \vee f(b) \scalarMult \exists(a)$.
 
 This implies $f(a \mathop{\widetilde +} b) \le f(a) \vee f(b)$.
 Now multiplying this by $\exists(a)\scalarMult\exists(b)$, we have $f(a \mathop{\widetilde +} b) = f(a \mathop{\widetilde +} b) \scalarMult\exists(a) \scalarMult\exists(b) \le (f(a) \vee f(b)) \scalarMult \exists(a) \scalarMult \exists(b) = f(a) \scalarMult \exists(b) \vee f(b) \scalarMult \exists(a)$, since $a = a \scalarMult \exists(a)$.
 Thus, these two conditions are equivalent.
 
 In summary, the elements of $\OPAI_R(Q)$ correspond to quantale homomorphisms $f\colon \MonIdl R \to Q$ satisfying $f(\widetilde{0}_+) = 0$ and $f(a \mathop{\widetilde +} b) \le f(a) \vee f(b)$,
 but these are in turn in bijection with quantale homomorphisms out of the quotient $\Idl(R)$ by \cref{prop:kernel_congruence_of_addquot}, as required.
\end{proof}

Setting $\Rad(R)$ to be the localic reflection of $\Idl(R)$, we obtain the following corollary.

\begin{corollary}\label{prop:localic_spectrum_for_localic_semirings}
 Let $R$ be an overt commutative localic semiring and suppose $\Sats R$ is a dualisable suplattice.
 Then $\OPAI_R\colon \Loc\op \to \Set$ is representable with representing object $\Rad(R)$ and universal element $(\rad\addquot \otimes \kappa)\eta(\top)$,
 where $\rad\colon \Idl(R) \twoheadrightarrow \Rad(R)$ is the unit of the localic reflection.
\end{corollary}

We can also prove an analogue of \cref{prop:when_localic_monoid_spectrum_exists}.
\begin{theorem}\label{prop:when_localic_semiring_spectrum_exists}
 Let $R$ be an overt commutative localic semiring and suppose $\Sats R$ is locally compact. Then $\OPAI_R\colon \Loc\op \to \Set$ is representable.
\end{theorem}
\begin{proof}
 We use a presentation of the localic spectrum in terms of $\Sats R$ as in \cref{prop:when_localic_monoid_spectrum_exists}.
 The additive relations are given by the modified operations as in \cref{prop:quantic_spectrum_for_localic_semirings}.
\end{proof}
Note that while \cref{prop:when_localic_semiring_spectrum_exists} is more general than \cref{prop:localic_spectrum_for_localic_semirings}, it does not tell us the precise form of the spectrum.
Whether the spectrum in \cref{prop:when_localic_semiring_spectrum_exists} is still given by the frame of radical ideals is a possible topic for further research.

In future it could also be interesting to consider $\hom(\O R, \Omega)$ (and hence $\Idl(R)$) as a \emph{localic} suplattice as in \cite{resende2003localic} instead of simply a suplattice.
However, this would require the development of a theory of quotients of localic suplattices.

\section{Special cases}

In this section we will show that our construction reduces to the known ones in each of our original cases and describe some unusual examples.
Let us start with a lemma that will be helpful for showing when the conditions of \cref{prop:quantic_spectrum_for_localic_semirings} hold.

\begin{lemma}\label{prop:equivalent_conditions_for_SR_dualisable}
 Let $R$ be an overt localic semiring. The following conditions are equivalent:
 \begin{enumerate}
  \item $\Sats R$ is a dualisable suplattice,
  \item There is a family $(r_x)_{x \in X}$ of elements of $\O R$ and a family $(W_x)_{x \in X}$ of overt weakly closed sublocales of $R$ such that $s = \bigvee_{W_x \between s} r_x$
        for all saturated elements $s \in \O R$,
  \item $u \le \bigvee_{V \between u} \stabpos(V)$ for all $u \in \O R$, where $\stabpos(V) = \bigwedge\{ t \text{\textup{ saturated}} \mid V \between t \}$.
 \end{enumerate}
 If $R$ satisfies these equivalent conditions we say it is \emph{approximable}.
\end{lemma}
\begin{proof}
 The equivalence of $(i)$ and $(ii)$ follows easily from \cref{prop:supercontinuous_vs_dualisable} by composing with $\kappa$ and $\kappa_!$ as appropriate.
 
 To see that $(ii)$ implies $(iii)$ observe that $\bigvee_{W_x \between s} r_x \le s$ implies $W_x \between s \implies r_x \le s$ for all saturated opens $s$ and
 hence, $r_x \le \stabpos(W_x)$. So for saturated $u$ we can immediately conclude that $u = \bigvee_{W_x \between u} r_x \le \bigvee_{V \between u} \stabpos(V)$.
 Now recall that $V \between (\iota_1)_! \mu_\times(t) \iff V \top \between t$. It follows that $\stabpos(V) = \stabpos(V \top)$ and
 $u \le (\iota_1)_! \mu_\times(u) \le \bigvee_{V \between (\iota_1)_! \mu_\times(u)} \stabpos(V) = \bigvee_{V \top \between u} \stabpos(V) =
 \bigvee_{V \top \between u} \stabpos(V \top) \le \bigvee_{V \between u} \stabpos(V)$ for general $u$, as required.
 
 Finally, we show $(iii)$ implies $(ii)$. We let $W_x$ range over all overt weakly closed sublocales and set $r_x = \stabpos(W_x)$.
 Then $(iii)$ gives $s \le \bigvee_{W_x \between s} \stabpos(W_x)$, while $\bigvee_{W_x \between s} \stabpos(W_x) \le s$ follows for saturated $s$
 by the definition of $\stabpos(W_x)$.
\end{proof}

\begin{corollary}\label{prop:supercontinuous_implies_approximable}
 If $\O R$ is supercontinuous, then $\Sats R$ is a dualisable suplattice.
\end{corollary}
\begin{proof}
 By \cref{prop:supercontinuous_vs_dualisable}, $\O R$ has dual basis. This implies condition $(ii)$ above.
\end{proof}

\subsection{The discrete case}

The following proposition exhibits the Zariski spectrum of a discrete ring and the Stone spectrum of a distributive lattice as a special case of our construction.

\begin{proposition}
 Let $R$ be a discrete semiring. Then $R$ is approximable and the quantic spectrum is given by the usual quantale of (set-theoretic) ideals of $R$.
 The localic spectrum is then the usual pointfree Zariski spectrum of $R$.
\end{proposition}
\begin{proof}
 It is well known that discrete locales are supercontinuous and hence $\Sats R$ is dualisable by \cref{prop:supercontinuous_implies_approximable}.
 Then \cref{prop:quantic_spectrum_for_localic_semirings} implies that the localic spectrum is given by the quantale of overt weakly closed ideals.
 But the overt weakly closed sublocales of a discrete locale are precisely the open sublocales, which in turn correspond to its subsets. It is then
 easy to see that the overt weakly closed ideals coincide with the familiar notion of semiring ideals and the result follows.
\end{proof}

\subsection{The Hofmann--Lawson spectrum}

To obtain the Hofmann--Lawson spectrum, we equip the continuous frame $\O X$ in question with the Scott topology to obtain a spatial localic distributive lattice $L$.
Note that $L$ is a supercontinuous locale (for instance, by observing that $L = \Srpnsk^X$ is a regular-injective and using \cite{banaschewski1991projective}) and hence its localic spectrum exists.
To determine the form of this spectrum, we will need the following lemma.

\begin{proposition}\label{prop:scott_closed_weakly_closed}
 Let $P$ be a continuous dcpo. The Scott-closed subsets of $P$ correspond to overt weakly closed sublocales in the Scott topology.
\end{proposition}
\begin{proof}
 Given a Scott-closed set $S \subseteq P$, we can define a suplattice map $h_S\colon \O P \to \Omega$ by $h_S(U) = \top \iff S \between U$.
 
 Conversely, given a suplattice homomorphism $h\colon \O P \to \Omega$, we define a set $S_h \subseteq P$ as the Scott-closure of $\{ x \in P \mid h(\twoheaduparrow x) = \top\}$.
 (This set is already downwards closed, but needs to be closed under directed suprema.)
 
 We now show these are inverse operations. Take $x \in S$. If $y \ll x$, then $S \between \twoheaduparrow y$ and so $y \in S_{h_S}$.
 But $x = \dirsup \twoheaddownarrow x$ by continuity and hence $x \in S_{h_S}$ by Scott-closedness. Thus, $S \subseteq S_{h_S}$.
 Conversely, if $S \between \twoheaduparrow x$ then $x \in S$ by downward closure and hence $S_{h_S} \subseteq S$, since $S$ is Scott-closed.
 
 Suppose $h(U) = \top$. We may write $U = \bigcup_{x \in U} \twoheaduparrow x$ by \cref{prop:base_scott_top_continuous_dcpo}
 and so $h(\twoheaduparrow x) = \top$ for some $x \in U$. Thus, $S_h \between U$ and hence $h_{S_h}(U) = \top$. Consequently, $h \le h_{S_h}$.
 On the other hand, if $S_h \between U$, then $\{ x \in P \mid h(\twoheaduparrow x) = \top\} \between U$ by \cref{prop:scott_closure_continuous_dcpo} and the Scott-openness of $U$.
 Thus, $h(\twoheaduparrow x) = \top$ for some $x \in U$ so that $h(U) = \top$. Hence, $h_{S_h} \le h$ and the assignments are inverse operations.
\end{proof}

\begin{corollary}
 The localic spectrum of a continuous frame $L \cong \O X$ equipped with the Scott topology is isomorphic to $X$ and hence coincides with the Hofmann--Lawson spectrum.
\end{corollary}
\begin{proof}
 By \cref{prop:quantic_spectrum_for_localic_semirings} the quantic spectrum is given by the quantale of overt weakly closed ideals. From \cref{prop:scott_closed_weakly_closed} we
 find that these correspond to the Scott-closed ideals of $L$, which are precisely the principal downsets. We can then see that the quantic spectrum is given by the frame $L \cong \O X$.
 Since this is already a frame, this is also the frame of the localic spectrum of $L$.
\end{proof}

\subsection{The Gelfand spectrum}

It remains to show that the Gelfand spectrum is a special case of our construction. For simplicity, we will show this for the locale of real-valued functions on a compact regular locale.
The complex case is essentially identical. It should also be possible to prove it directly from the axioms of a commutative localic C*-algebra as given in \cite{Henry2016}.

\begin{proposition}\label{prop:gelfand_approximable}
 Suppose $X$ is a compact regular locale. Then $A = \R^X$ is an overt approximable localic ring.
\end{proposition}
\begin{proof}
 It is shown in \cite{Henry2016} that $A$ is overt and furthermore a localic Banach algebra. 
 By \cref{prop:equivalent_conditions_for_SR_dualisable} we must show that $u \le \bigvee_{V \between u} \stabpos(V)$ for all $u \in \O A$.
 
 Given an overt sublocale $F$ of $A$ and a rational $\epsilon > 0$, we define $B_\epsilon(F)$ by the expression $\{h \colon A \mid \exists f\colon F.\ \norm{h-f} < \epsilon\}$ in the internal logic
 of the coherent hyperdoctrine of open sublocales.
 By \cite{Henry2016}, we may express $u$ as $u = \bigvee_i F_i$ where each $F_i$ is a positive open sublocale such that $B_{\epsilon_i}(F_i) \le u$ for some rational $\epsilon_i > 0$.
 Furthermore, each $F_i$ can then in turn be expressed as a join of the open sets contained in it of diameter less than $\tfrac{1}{4}\epsilon_i$.\footnote{A sublocale $F$ has diameter less than a
 positive rational $\delta$ if $\,\vdash_{x,y \colon F} \norm{x-y} < \delta$ in the internal logic.}
 Thus, we may assume $\diam(F_i) < \frac{1}{4}\epsilon_i$ without loss of generality.
 
 Consider $F = F_i$ and $\epsilon = \epsilon_i$ for some $i$, where we omit the subscripts for notational convenience.
 The idea is to construct an overt sublocale $G$ such that $G \between B_\epsilon(F)$ and every saturated open $t$ which meets $G$ contains $F$.
 If we can do this, we could deduce $F \le \bigwedge\{t \in \Sats A \mid G \between t\} = \stabpos(G)$. And clearly $G \between u$,
 since $G \between B_\epsilon(F) \le u$.
 Therefore, $F \le \bigvee_{V \between u} \stabpos(V)$. The result would then follow by taking the join over all such $F$.

 We now describe how to construct $G$. The intuition behind our construction is that we want to find functions which are sufficiently close to those in $F$,
 but which have inflated zero sets. The absolute value map and the meet and join operations on $\R$ induce `pointwise' operations on $A$ which satisfy
 $h \vee 0 = \tfrac{1}{2}(h + \abs{h})$ and $h \wedge 0 = \tfrac{1}{2}(h - \abs{h})$ in the internal logic, where $0$ is the constant zero map. Note that the map defined by $f \mapsto (f \vee 0) + (f \wedge 0)$ in the internal logic is simply the identity.
 We define a locale map $\zeta\colon F \to A$ by $f \mapsto ([f - \tfrac{1}{2} \epsilon] \vee 0) + ([f + \tfrac{1}{2} \epsilon] \wedge 0)$
 and set $G$ to be the image of $F$ under $\zeta$.
 
 Let us show that $G \between B_\epsilon(F)$. Note that $G \between B_\epsilon(F) \iff F \between \zeta^*(B_\epsilon(F))$, or in the internal logic,
 $\exists f\colon F.\ \zeta(f) \in B_\epsilon(F)$. So by the definition of $B_\epsilon(F)$, we must show $\exists f\colon F.\ \exists f'\colon F.\ \norm{\zeta(f) - f'} < \epsilon$.
 A straightforward calculation in the internal logic gives $\norm{\zeta(f) - f} = \tfrac{1}{2}\norm{\abs{f+\tfrac{1}{2}\epsilon} - \abs{f-\tfrac{1}{2}\epsilon}} \le
 \tfrac{1}{2}\norm{(f+\tfrac{1}{2}\epsilon) - (f-\tfrac{1}{2}\epsilon)} = \tfrac{1}{2}\norm{\epsilon} = \tfrac{1}{2}\epsilon$,
 where the inequality is from the reverse triangle inequality. (Here the judgemental inequality relation can be defined from equality and $\wedge$ in the usual way.)
 Thus, $\norm{\zeta(f) - f} < \epsilon$. Now since $F$ is positive, we have $\exists f\colon F.\ \top$. Combining these yields $\exists f\colon F.\ \norm{\zeta(f) - f} < \epsilon$ and the desired result follows.
 
 Now we show that every saturated open $t$ which meets $G$, contains $F$.
 The plan is to construct a localic map $\psi\colon F\times F \to A$ such that $\zeta(f) = h \times \psi(f,h)$ in the internal logic.
 
 Given such a map $\psi$, we can conclude the result as follows.
 Since $G \between t$, we know $\exists f \colon F.\ \zeta(f) \in t$.
 Now since $\zeta(f) = h \times \psi(f,h)$, we have $\exists f \colon F.\ h \times \psi(f,h) \in t$ for $h\colon F$.
 But $t$ being saturated means $\exists y\colon A.\ xy \in t \vdash_{x \colon A} x \in t$.  Putting these together,
 we may conclude $\vdash_{h \colon F} h \in t$, which gives $F \le t$ as desired.
 
 We define $\psi\colon F \times F \to \R^X$ by specifying its uncurried form $\psi^\flat\colon F \times F \times X \to \R$
 on a covering pair of opens: $W = \{(f, h, x) \colon F \times F \times X \mid \abs{f(x)} > \tfrac{1}{4}\epsilon \}$
 and $Z = \{(f, h, x) \colon F \times F \times X \mid \abs{f(x)} < \tfrac{1}{2}\epsilon \}$.
 
 Note that restricted evaluation map $\ev(\pi_2,\pi_3)\vert_W$ factors through $\R^* = (-\infty, 0) \vee (0, \infty)$.
 This can be seen using internal logic: consider a triple $(f,h,x)\colon W$.
 Because $\diam(F) < \tfrac{1}{4}\epsilon$, we have $\norm{f - h} < \tfrac{1}{4}\epsilon$.
 By the definition of the norm, this gives $\abs{f(x)-h(x)} < \tfrac{1}{4}\epsilon$.
 Putting this together with $\tfrac{1}{4}\epsilon < \abs{f(x)}$, we find $\abs{f(x) - h(x)} < \abs{f(x)}$, or equivalently, $0 < \abs{f(x)} - \abs{f(x) - h(x)}$.
 But then $0 < \abs{f(x)} - \abs{f(x) - h(x)} \le \abs{f(x) - (f(x) - h(x))} = \abs{h(x)}$ by the reverse triangle inequality.
 We may conclude that $\ev(h,x) = h(x)$ lies in $\R^*$ as required.
 
 We define $\psi^\flat\vert_Z\colon (f,h,x) \mapsto 0$ and $\psi^\flat\vert_W\colon (f,h,x) \mapsto (h(x))^{-1} \times \zeta(f)(x)$ in the internal logic,
 where $(-)^{-1}\colon \R^* \to \R^*$ is the reciprocal operation and we implicitly factor $\ev(\pi_2,\pi_3)\vert_W$ through $\R^*$.
 To see that these agree on the overlap $W \wedge Z$, we show that $\zeta^\flat(\pi_1,\pi_3)$ is zero on $Z$.
 Recall that $\zeta(f) = ([f - \tfrac{1}{2} \epsilon] \vee 0) + ([f + \tfrac{1}{2} \epsilon] \wedge 0)$ in the internal logic. Now if $(f,h,x)$ lies in $Z$, we have
 $(f(x) < \tfrac{1}{2}\epsilon) \land (f(x) > - \tfrac{1}{2}\epsilon)$. Then $f(x) - \tfrac{1}{2}\epsilon < 0$ and $f(x) + \tfrac{1}{2}\epsilon > 0$. We then quickly see that $\zeta(f)(x) = 0$, as required.
 
 Finally, if $i$ is the inclusion of $F$ into $A$, we show $\mu_\times (i \pi_2, \psi) = \zeta \pi_1$.
 To see this we uncurry each expression and consider the two restrictions in the internal logic. We must show $0 = \zeta(f)(x)$ on $Z$ and
 $h(x) \times ((h(x))^{-1} \times \zeta(f)(x)) = \zeta(f)(x)$ on $W$. But we have already shown the first equality above and the second one follows
 immediately from associativity of multiplication and properties of the reciprocal.
\end{proof}

\begin{proposition}\label{prop:gelfand_expected_spectrum}
 If $X$ is a compact regular locale, then the localic spectrum of $\R^X$ is isomorphic to $X$ (as in Gelfand duality).
\end{proposition}
\begin{proof}
 By the localic Gelfand duality of \cite{Henry2016}, we know that the localic Gelfand spectrum of the C*-algebra $\C^X$ is isomorphic to $X$
 and it is shown in \cite{ConstructiveGelfandNonunital} that the opens of the Gelfand spectrum are in turn in bijection with the overt weakly closed ideals of $\C^X$.
 The use of complex numbers here is not essential and similar results hold for the real algebra $\R^X$ giving an order isomorphism $\Idl(\R^X) \cong \O X$.
 
 Now by \cref{prop:quantic_spectrum_for_localic_semirings,prop:gelfand_approximable} we have that the quantic spectrum of $\R^X$ is given by $\Idl(\R^X)$.
 It remains to show that $\Idl(\R^X)$ is a frame, from which it follows that the localic spectrum is given by $\Rad(\R^X) \cong \Idl(\R^X) \cong \O X$.
 
 The quantale $\Idl(\R^X)$ is a quotient of the quantale of monoid ideals $\MonIdl(\R^X)$ and so it is enough to show the latter is a frame.
 But $\MonIdl(\R^X)$ is a frame if and only if its dual $\Sats(\R^X)$ is a localic semilattice and it is not hard to show that this is in turn
 equivalent to requiring $f \in U \vdash_{f\colon \R^X} \exists k\colon \R^X.\ f^2 k \in U$ in the internal logic for all (basic) opens $U$.
 
 To prove this we can proceed in a very similar manner to the proof of \cref{prop:gelfand_approximable}.
 Consider $B_\epsilon(F)$ as in that proof and construct $\zeta$ and $G$ in the same way. We then define a map $\psi\colon F \to \R^X$
 such that $\zeta(f) = f^2 \times \psi(f)$ by setting $\psi(f)(x) = 0$ when $\abs{f(x)} < \tfrac{1}{2}\epsilon$ and $\psi(f)(x) = (f(x)^2)^{-1} \times \zeta(f)(x)$ when $\abs{f(x)} > \tfrac{1}{4}\epsilon$.
 Here $\psi(f)$ plays the role of the $k$ required to deduce the result.
\end{proof}

\subsection{Unusual examples}

Let us end with some examples which are rather unlike the ones we have seen.

\begin{example}
 Consider the locale $\lowerRealsNonneg$ of nonnegative lower reals --- this is given by the theory of (possibly empty) lower Dedekind cuts on the nonnegative rationals. This is a localic semiring with the usual notions of addition and multiplication.
 It is not hard to show this is overt and approximable and the quantic spectrum is $\Omega$ --- a single point.
 If a discrete semiring $R$ has a quantic spectrum of $\Omega$, then it is a `Heyting semi-field'
 --- that is, a `local semiring' in which an element is not invertible if and only if it is zero.
 However, even classically $\lowerRealsNonneg$ is not a semi-field, since it contains the point $\infty$, which has no inverse. It is, however, the frame of opens of the non-sober topological space obtained by
 equipping $[0, \infty)$ with the topology of lower semicontinuity. This is a `paratopological semi-field' --- that is, the subspace $(0,\infty)$ of nonzero elements is a topological monoid under multiplication
 for which every element has an inverse.
\end{example}
\begin{example}
 Consider the Sierpiński locale $\Srpnsk$ with the \emph{reverse} of its usual distributive lattice structure. This is overt and approximable,
 but its frame of radical ideals is isomorphic to the trivial frame and so its spectrum is empty, despite the semiring being nontrivial.
 
 This example suggests the possibility of a `finer' spectrum than the one described here. Indeed, for a spectrum of rings it is natural to consider ideals, but for more general semirings
 it does not seem too surprising that this might fail to capture some important information. This was not a problem for continuous frames with the Scott topology, where the topology, algebraic and order-theoretic
 properties work in concert, but perhaps this is less true in the general case. We do believe that our spectrum still gives some useful information even for general semirings,
 but it could be interesting to consider a richer notion of semiring spectrum in future.
\end{example}

It appears that the spectrum can even sometimes fail to exist at all.

The forgetful functor from $\Frm$ to $\Dcpo$ has a left adjoint, which induces a comonad on $\Frm$ and hence a monad $\doublePow$ on $\Loc$.
In \cite{VickersTownsendDoublePowerlocale} it is shown that $\doublePow$ sends a locale $X$ to the double exponential $\Srpnsk^{\Srpnsk^X}$,
even in the case that $X$ is not exponentiable, so long as the intermediate exponential is taken in $\Set^{\Loc\op}$.
We have shown that if $X$ is an exponentiable locale, then $X \cong \Spec(\Srpnsk^X)$. This suggests the following conjecture.

\begin{conjecture}
There is a natural isomorphism $\OPAI_{\doublePow X} \cong \Hom_\Loc((-) \times X, \Srpnsk)$.
\end{conjecture}
In particular, this would imply that $\Spec(\doublePow X)$ fails to exist whenever $X$ is not locally compact.

\hypersetup{bookmarksdepth=1}
\subsection*{Acknowledgements}

I would like to thank Peter Faul for discussions we had on the presentation of this paper.

\bibliographystyle{abbrv}
\bibliography{references}

\end{document}